\documentclass[12pt]{amsart}

\usepackage[utf8]{inputenc}
\usepackage{graphicx}
\usepackage{graphicx,color}
\usepackage{latexsym}
\usepackage[all]{xy}
\usepackage{mathrsfs}
\usepackage{enumerate}
\usepackage[shortlabels]{enumitem}
\usepackage{multicol}
\usepackage{lipsum}
\usepackage{amssymb}
\usepackage{tikz}
\usetikzlibrary{cd}
\usetikzlibrary{patterns}
\definecolor{DimGray}{rgb}{0.41, 0.41, 0.41}
\usepackage{float}
\usepackage{bm}
\usepackage{mathptmx}
\usetikzlibrary{shapes.geometric}
\usepackage{dsfont}
\usepackage{cjhebrew}
\usepackage[alpine]{ifsym}
\usepackage{xpicture}
\usepackage{calculator}
\usepackage{graphicx}
\usepackage{pgf,tikz,pgfplots}
\pgfplotsset{compat=1.15}
\usetikzlibrary{arrows}

\usepackage{amsmath}
\usepackage{amsfonts}
\usepackage{amssymb,enumerate}
\usepackage{amsthm}
\usepackage{fancyhdr}
\usepackage{hyperref}
\usepackage{cleveref}

\usepackage{caption}
\usepackage{subcaption}

\def\sideremark#1{\ifvmode\leavevmode\fi\vadjust{\vbox to0pt{\vss 
      \hbox to 0pt{\hskip\hsize\hskip1em           
 \vbox{\hsize2cm\tiny\raggedright\pretolerance10000
 \noindent #1\hfill}\hss}\vbox to8pt{\vfil}\vss}}} %
\theoremstyle{plain}
\newtheorem{theorem}{Theorem}[section]
\newtheorem{proposition}[theorem]{Proposition}

\newtheorem*{theorem*}{Theorem}
\newtheorem{lemma}[theorem]{Lemma}
\newtheorem{corollary}[theorem]{Corollary}

\theoremstyle{definition}
\newtheorem{defin}[theorem]{Definition}
\newtheorem{question}[theorem]{Question}

\newtheorem{ex}[theorem]{Example}

\theoremstyle{remark}
\newtheorem{rem}[theorem]{Remark}

\def\lcm{\operatorname{lcm}}

\counterwithin{figure}{section}
\numberwithin{equation}{section}

\title[Miniversal deformations of certain complete intersection monomial curves]{
The miniversal deformation of certain complete intersection monomial curves}

\author{Patricio Almir\'on}

\author{Julio-Jos\'e Moyano-Fern\'andez}



\thanks{The first author is supported by Spanish Ministerio de Ciencia, Innovaci\'{o}n y Universidades PID2020-114750GB-C32 and by the IMAG–Maria de Maeztu grant CEX2020-001105-M / AEI /10.13039/501100011033, through a postdoctoral contract in the ‘Maria de Maeztu Programme for Centres of Excellence’. The second author was partially funded by MCIN/AEI/10.13039/501100011033, by “ERDF A way of
making Europe” and by “European Union NextGeneration EU/PRTR”, Grants PID2022-138906NB-C22 and
TED2021-130358B-I00, as well as by Universitat Jaume I, Grants UJI-B2021-02 and GACUJIMA-2023-06.}

\address{Instituto de Matemáticas\\
Universidad de Granada\\
18001, Granada, Spain.}
\email{patricioalmiron@ugr.es}

\address{Universitat Jaume I, Campus de Riu Sec, Departamento de Matem\'aticas \& Institut Universitari de Matem\`atiques i Aplicacions de Castell\'o, 12071
Caste\-ll\'on de la Plana, Spain}

\email{moyano@uji.es}

\begin{document}

\maketitle
\begin{abstract}
The aim of this paper is to provide an explicit basis of the miniversal deformation of a monomial curve defined by a free semigroup---these curves make up a notable family $\mathcal{C}$ of complete intersection monomial curves. First, we dispense a general decomposition result of a basis $B$ of the miniversal deformation of any complete intersection monomial curve. As a consequence, we explicitly calculate $B$ in the particular case of a monomial curve defined from a free semigroup. This direct computation yields some estimates for the dimension of the moduli space of the family $\mathcal{C}$.
\end{abstract}
\section{Introduction}
 Deformations are an intuitive but nevertheless rich tool in the study of curve singularities. A particular type of curve singularities is the class of monomial curves, i.e. those defined through a parameterization  \(C:(t^{a_0},\dots,t^{a_g})\) where \(a_0,\dots,a_g\in\mathbb{N}\) are such that \(\gcd(a_0,\dots,a_g)=1.\) Therefore, the semigroup algebra associated to the curve coincides with the semigroup algebra defined by the numerical semigroup \(\Gamma=a_1\mathbb{N}+\cdots+a_g\mathbb{N}=\langle a_1,\ldots , a_g \rangle.\) We write $C^\Gamma$ for the monomial curve defined by $\Gamma$. In the particular case in which \(C\) is a complete intersection curve singularity, Delorme \cite{delormeglue} provides a recursive method to compute the implict equations \(f_1,\dots,f_g\) of \(C\) depending on the combinatorics and relations between the generators of \(\Gamma.\) In this way, it is natural to ask for an explicit description of the monomial basis of the miniversal deformation in terms of the combinatorics of the numerical semigroup \(\Gamma.\)
\medskip

For complete intersection monomial curve singularities \(C^\Gamma\), it is well-known that the basis of the miniversal deformation is isomorphic as a \(\mathbb{C}\)--vector space to the Tjurina algebra
\[
T^1(C^\Gamma):=\frac{\mathbb{C}[u_0,\dots,u_g]^g}{\left(\frac{\partial f_i}{\partial u_j}\right)_{i,j}\mathbb{C}[u_0,\dots,u_g]^{g+1}+(f_1,\dots,f_g)\mathbb{C}[u_0,\dots,u_g]^g}.
\]
(We will sometimes write just $T^1$ for brevity). This algebra can be endowed naturally with a grading. The corresponding graded components were described by Buchweitz \cite{Buchweitz} in terms of a rather intricate combinatorial formula, even in the complete intersection case. To the best of the authors' knowledge, it is remarkable that Delorme's recursive representation of implicit equations has not been yet employed in the literature to describe the combinatorial aspects of those graded components, nor even to provide an explicit \(\mathbb{C}\)--basis. Recognizing the potential utility of this approach, our main goal is to make full use of Delorme's description for this purpose in order to recursively compute a monomial basis of \(T^1(C^\Gamma).\) First, we observe that Delorme's method provides an order in the implicit equations such that the Jacobian matrix has a block decomposition in terms of the gradient vector of \(f_g,\) the Jacobian matrix of \(f_1,\dots,f_{g_1}\) and the Jacobian matrix of \(f_{g_1+1},\dots,f_{g_1+g_2}\) in such a way that  \(g_1+g_2=g-1.\) Moreover, \(f_1,\dots,f_{g_1}\) and \(f_{g_1+1},\dots,f_{g_1+g_2}\) define two different complete intersection monomial curves \(C_{1},C_{2}.\) Our first main result is that the bases of the miniversal deformations of \(C_{1}\) and \(C_{2}\) are embedded in the basis of the miniversal deformation of \(C^{\Gamma}\) (Theorem \ref{thm:injective}); this result is in fact a generalization of Cassou-Nogu\`es \cite[Th\'eor\`eme~2]{pierette}.
\medskip

Though Delorme's algorithm might initially suggest the capability to achieve the computation of a monomial basis of the miniversal deformation of a complete intersection monomial curve, regrettably we have encountered limitations that hinder the attainment of complete generality in this task. However, imposing some extra conditions on the generators of the complete intersection semigroup one can succeed in the description of the basis. This has been already shown by Cassou-Nogu\`es \cite[Th\'eor\`eme~3]{pierette} in the particular case of a semigroup satisfying the properties
\begin{enumerate}
    \item  $n_ia_i\in a_0\mathbb{N}+a_1\mathbb{N}+\cdots + a_{i-1}n_{i-1}\mathbb{N}$, \ for every $i=1,\ldots , g$;
    \item $n_{i-1}a_{i-1} < a_i$  \ for every $i=1,\ldots , g-1,$
\end{enumerate}
where $n_0=1$ and $n_i=\mathrm{gcd}(a_0,a_1,\ldots , a_{i-1})/\mathrm{gcd}(a_0,a_1,\ldots , a_i)$ for $i=1,\ldots , g$. This kind of numerical semigroups are called plane curve numerical semigroups, as they occurr as semigroups associated to germs of irreducible plane curve singularities, see Zariski \cite{Zarbook}. For this class, she provides the monomial basis in terms of the generators of the semigroup \cite[Th\'eor\`eme~3]{pierette}. Our second main result of this paper (Theorem~\ref{thm:basisdeffree}) yields precisely a generalization of \cite[Th\'eor\`eme~3]{pierette} eliminating the hypothesis of $n_{i-1}a_{i-1} < a_i$  \ for every $i=1,\ldots , g-1$. This leads us to consider only the condition $n_ia_i\in a_0\mathbb{N}+a_1\mathbb{N}+\cdots + a_{i-1}n_{i-1}\mathbb{N}$, \ for every $i=1,\ldots , g$, which gives a class of numerical semigroups called free semigroups (see Section \ref{sec:miversalCI}) larger than the class of plane curve numerical semigroups.
\medskip

An important topic related to the study of the base space of the miniversal deformation of a monomial curve is its connection with the moduli space of projective curves with a given Weierstra\ss~semigroup. For any kind of singularity, the module \(T^1\) can be properly defined by using Schlessinger's formal deformation theory \cite{Sch} (see also the book of Greuel, Lossen and Shustin \cite{Greuelbook}). For monomial curve singularities \(C^\Gamma\), the module \(T^1(C^\Gamma)\) has a natural grading as in the particular case of a complete intersection. Pinkham \cite[Theorem~13.9]{Pinkham1} shows that the negatively graded part of \(T^1(C^\Gamma)\) is strongly linked to the moduli space of projective curves with a prescribed Weierstra\ss~semigroup \(\Gamma\). As first order deformations of a complete intersection singularity are unobstructed, the dimension of the moduli space can be computed from the negative part of \(T^1(C^\Gamma)\). Prior outcomes due to St\"ohr \cite{stohr1}, Contiero and St\"ohr \cite{sthor2}, Stevens \cite{stevens}, or Rim and Vitulli \cite{RV77} have not furnished us with explicit formulas in terms of the generators of the semigroup. Our aim is to iteratively compute the dimension of the moduli space associated to a free semigroup by employing Theorem~\ref{thm:basisdeffree}. Once again, we find it imperative to establish precise lower and upper bounds for this recursive computation, because the intricate combinatorial aspects require a rigorous control in achieving a complete closed formula, as we will see in Section \ref{sec:dimensionmoduli}. The main result of Section \ref{sec:dimensionmoduli} is precisely Theorem \ref{thm:dimensionpositive} which allows us to estimate the dimension of the negative part of \(T^1\) in terms of the generators of the semigroup. Moreover, the bounds are sharp and, in the case of an irreducible plane curve semigroup, Theorem \ref{thm:dimensionpositive} recovers \cite[Th\'eor\`eme~8]{pierette} (see also Proposition \ref{prop:Proposition1}). In addition, Theorem \ref{prop:charirreduciblesemigroup} yields a characterization of a plane curve semigroup in terms of the dimension of the positive part of $T^1$, which appears to be unknown in the literature to our best knowledge. 
\medskip

Besides the plane curve semigroups, a further relevant family of free semigroups are those associated to curves with only one place at infinity, cf.~Abhyankar and Moh \cite{AMoh}, Galindo and Monserrat \cite{GalMonJalg}, say semigroups at infinity for brevity. In this case, we investigate also the moduli space, as said before, and we have been able to compute exactly the dimension of the positive part of the corresponding $T^1$ for the case of semigroups at infinity with three generators in Theorem \ref{th:moduliatinfinity}. Attached to a singular point of an algebraic curve, we can define therefore both the plane curve semigroup and the semigroup at infinity: these describe two different ways to investigate the same singularity so that they are related e.g. through their respective sets of generators \cite[Proposition 2.1]{GalMonJalg}. In \cite{teissierappen}, Teissier showed that the positively graded part of the miniversal deformation of a plane curve semigroup encodes the analytic moduli associated with an irreducible plane curve with a fixed topological type. As we will see in Section \ref{sec:4}, our Theorem \ref{thm:basisdeffree} and the relation between the generators of the semigroup at infinity and the semigroup of the curve establish a sort of duality between the moduli space in the sense of Pinkham associated to the semigroup at infinity with the analytic moduli in the sense of Teissier associated to the semigroup of the curve---astonishingly not investigated so far. We therefore conclude the paper with a final discussion on this issue.
\medskip

\noindent \textbf{Acknowledgements}. The authors wish to express their gratitude to Carlos-Jesús Moreno-Ávila for many stimulating and useful conversations during the preparation of the paper; he provided us with Example \ref{ex:examplebasis} and the \(\delta\)-sequence of Example \ref{example:Zariskiinvariant} as well.



\medskip


\section{Deformations of complete intersection monomial curves}\label{sec:miversalCI}

Let \(\Gamma=\langle a_0,\dots,a_g\rangle \) be a numerical semigroup.  Let \(t\in\mathbb{C}\) be a local coordinate of the germ \((\mathbb{C},0)\) and let \((u_0,\dots,u_g)\in \mathbb{C}^{g+1}\) be local coordinates of the germ \((\mathbb{C}^{g+1}, \boldsymbol{0})\). The monomial curve \( (C^\Gamma, \boldsymbol{0}) \subset (\mathbb{C}^{g+1}, \boldsymbol{0}) \) defined via the parameterization
\[ 
C^\Gamma : u_i = t^{a_i}, \qquad i=1,\ldots, g
\]

is called the monomial curve associated to \(\Gamma\). Write $\mathbb{C}[C^{\Gamma}]:=\mathbb{C}[t^{\nu}\;:\;\nu\in\Gamma],$ which coincides with the semigroup algebra \(\mathbb{C}[\Gamma]\) associated to $\Gamma$. We will use either notation depending on whether we want to highlight the geometric or algebraic interpretation. The numerical semigroup $\Gamma$ is said to be complete intersection if \(\mathbb{C}[\Gamma]\) is a complete intersection; this means that, if \(\Gamma=\langle a_0,\dots,a_g\rangle\) is generated by \(g+1\) elements, then we have an exact sequence
\begin{equation}\label{eq:curvemonomialalgebra}
	0\rightarrow \mathbb{C}[u_0,\dots,u_g]^g\rightarrow\mathbb{C}[u_0,\dots,u_g]\xrightarrow{\varphi} \mathbb{C}[C^{\Gamma}]\rightarrow 0,
\end{equation}
with \(\operatorname{ker}\varphi=(f_1,\dots,f_g)\) and \(f_1,\dots,f_g\) define a regular sequence. In particular, this implies that the monomial curve \((C^\Gamma, \boldsymbol{0}) \) is a complete intersection.
\medskip

In 1976, Delorme \cite[Lemme~7]{delormeglue} showed the following combinatorial characterization of a complete intersection numerical semigroup. Set $A:=\{a_0,\dots,a_g\}$; for \(\Gamma=\langle A\rangle\) we have that \(\Gamma\) is a complete intersection numerical semigroup if and only if there exists a partition \(A=A_1\sqcup A_2\) of the set of generators $A$ with $A_1 \neq \emptyset \neq A_2$ such that \(\mathbb{C}[\Gamma_{A_i/d_i}]\) are complete intersections defined by \(I_i:=\ker\varphi_i\), where \(d_i:=\gcd(A_i)\) for $i=1,2$, and \(\mathbb{C}[\Gamma]\) is defined by \(I_1+I_2+\rho\) with \(\deg(f_g)=\operatorname{lcm}(d_1,d_2).\) More precisely, if we set \(g_i:=|A_i|-1\), then \(g_1+g_2=g-1\) and we have the exact sequences
\begin{equation}\label{eqn:seqgluing}
\begin{split}
     0\rightarrow \mathbb{C}[x_0,\dots,x_{g_1}]^{g_1}\rightarrow\mathbb{C}[x_0,\dots,x_{g_1}]\xrightarrow{\varphi_1} \frac{\mathbb{C}[x_0,\dots,x_{g_1}]}{I_1}\simeq \mathbb{C}[C^{\Gamma_{A_1/d_1}}]\rightarrow0\\
    0\rightarrow \mathbb{C}[y_{0},\dots,y_{g_2}]^{g_2}\rightarrow\mathbb{C}[y_{0},\dots,y_{g_2}]\xrightarrow{\varphi_2} \frac{\mathbb{C}[y_0,\dots,y_{g_2}]}{I_2}\simeq \mathbb{C}[C^{\Gamma_{A_2/d_2}}]\rightarrow 0
\end{split}
\end{equation}

In addition, we can define a binomial \(\rho\) in separated variables with degree \(\lcm(d_1,d_2).\) Thus there is a natural decomposition of the semigroup algebra of \(\Gamma\) in the form
\[\mathbb{C}[\Gamma]=\frac{\mathbb{C}[\Gamma_{A_1/d_1}]\otimes \mathbb{C}[\Gamma_{A_2/d_2}]}{\rho },\]
where we write \(\mathbb{C}[u_0,\dots,u_g]=\mathbb{C}[x_0,\dots,x_{g_1};y_0,\dots,y_{g_2}].\)

\begin{rem}
    The set of generators \(A\) of \(\Gamma\) does not need to be a minimal generating set. 
\end{rem}

On the other hand, given two complex space germs $(X,x)$ and $(S,s)$, a deformation of $(X,x)$ over $(S,s)$ consists of a flat morphism $\phi: (\mathcal{X},x)\to (S,s)$ of complex germs together with an isomorphism from $(X,x)$ to the fibre of $\varphi$, namely $(X,x)\to (\mathcal{X}_s,x):=(\phi^{-1}(s),x)$. Here, $(\mathcal{X},x)$ is called the total space, $(S,s)$ is the base space, and $(\mathcal{X}_s,x)\cong (X,x)$ is the special fibre of the deformation. A deformation is usually denoted by
$$
(i,\varphi): (X,x) \overset{i}{\rightarrow} (\mathcal{X},x) \overset{\phi}{\rightarrow} (S,s).
$$
Roughly speaking, a deformation is called versal if any other deformation results from it by a base change. A deformation is called miniversal if it is versal and the base space is of minimal dimension. The existence of a miniversal deformation for isolated singularities is a celebrated result by Grauert \cite{grauert}. The reader is refereed to \cite[Part II, Sect.~1]{Greuelbook} for further details.


\subsection{The miniversal deformation of a complete intersection monomial curve}
Let \(C^{\Gamma}: (t^{a_0},\dots,t^{a_g})\) be a monomial curve which is a complete intersection defined by \(I=(f_1,\dots,f_g)\). We can apply Tjurina's theorem \cite{Tjurina} which in this case states that the base space of the miniversal deformation of \((C^{\Gamma},\boldsymbol{0})\) is 
\[
T^1(C^\Gamma):=\frac{\mathbb{C}[u_0,\dots,u_g]^g}{\left(\frac{\partial f_i}{\partial u_j}\right)_{i,j}\mathbb{C}[u_0,\dots,u_g]^{g+1}+(f_1,\dots,f_g)\mathbb{C}[u_0,\dots,u_g]^g}.
\]
Moreover, Pinkham \cite{Pinkham1} (see also \cite{BG80,Buchweitz,Deligne,Rim72}) showed that, if \(C^\Gamma\) is a complete intersection monomial curve, then the dimension of \(T^1(C^\Gamma)\) as a \(\mathbb{C}\)--vector space equals the conductor of the semigroup, i.e. \(\dim_\mathbb{C}T^1(C^\Gamma)=c(\Gamma):=\min\{\nu\in \Gamma\ |\ \nu+\mathbb{N}\subset \Gamma\}.\) Even more, from a result of Greuel \cite{Greuel75} as \(C^\Gamma\) is a quasi-homogeneous complete intersection singularity then \(c(\Gamma)=\dim_\mathbb{C}T^1(C^\Gamma)=\mu(C^\Gamma)\) where \(\mu(C^\Gamma)\) is the Milnor number (see also \cite{BG80}). From the decomposition of the semigroup algebra it is easily deduced that the Jacobian matrix presents a simple-to-describe block decomposition. Indeed, if we set \(\Gamma_i:=\Gamma_{A_i/d_i},\) \((h^1_1,\dots,h^1_{g_1})=I_1\) and \((h^2_1,\dots,h^2_{g_2})=I_2\), then the Jacobian matrix of the defining equations of \(C^\Gamma\) has a block decomposition in terms of the Jacobian matrices of \(C^{\Gamma_i}\) and an extra row in terms of the extra new relation as follows:

 \[\left(\frac{\partial f_i}{\partial u_j}\right)_{\tiny \begin{array}{c}
      1\leq i\leq g \\
      0\leq j\leq g 
 \end{array}}=\left(\begin{array}{cc}
     \left(\frac{\partial h^1_i}{\partial x_j}\right) & 0 \\
     0 & \left(\frac{\partial h^2_i}{\partial y_j}\right)
 \\
      \rho_1& \rho_2
 \end{array}\right ),\]
 where we identify \(f_1=h_1^1,f_2=h_2^1,\dots,f_{g_1}=h_{g_1}^1,\) \(f_{g_1+1}=h_1^2,f_{g_1+2}=h_2^2,\dots,f_{g_1+g_2}=h_{g_2}^2\) and \(f_g=\rho_1+\rho_2\) for \(\rho_1=(\partial\rho/\partial u_0,\dots,\partial\rho/\partial u_{g_1})\) and \(\rho_2=(\partial\rho/\partial u_{g_1+1},\dots,\partial\rho/\partial u_{g}).\)
 The consideration of this block decomposition leads to the proof of the fact that the base spaces of the miniversal deformations of \(C^{\Gamma_i} \) are contained in the base space of the miniversal deformation of \(C^{\Gamma}.\) 
 \medskip
 
 To do so, first observe that \(\left(\frac{\partial f_i}{\partial u_j}\right)_{i,j}\mathbb{C}[u_0,\dots,u_g]^{g+1}\) is just the submodule \[\overline{N}_\Gamma=\left\langle\Big(\frac{\partial f_1}{\partial u_0},\dots,\frac{\partial f_g}{\partial u_0}\Big),\dots,\Big(\frac{\partial f_1}{\partial u_g},\dots,\frac{\partial f_g}{\partial u_g}\Big)\right\rangle\subset \mathbb{C}[u_0,\dots,u_g]^{g}.\]

 Let us now define the submodules
 \[
\begin{split}
\overline{N}_{\Gamma_1}:=\left\langle\left(\frac{\partial h^1_1}{\partial x_0},\dots,\frac{\partial h^1_{g_1}}{\partial x_0}\right),\dots,\left(\frac{\partial h^1_1}{\partial x_{g_1}},\dots,\frac{\partial h^1_{g_1}}{\partial x_{g_1}}\right)\right\rangle\subset\mathbb{C}[x_0,\dots,x_{g_1}]^{g_1},&\\
\overline{N}_{\Gamma_2}:=\left\langle\left(\frac{\partial h^2_1}{\partial y_0},\dots,\frac{\partial h^2_{g_2}}{\partial y_0}\right),\dots,\left(\frac{\partial h^2_1}{\partial y_{g_2}},\dots,\frac{\partial h^2_{g_2}}{\partial xy_{g_2}}\right)\right\rangle\subset\mathbb{C}[y_0,\dots,y_{g_2}]^{g_2}.&
\end{split}
\]
 For each \(i=1,2\) we have the canonical projections
\[\tau_1:\mathbb{C}[\Gamma]^g=\bigoplus_{j=1}^{g}\mathbb{C}[\Gamma]e_j\rightarrow \mathbb{C}[\Gamma]^{g_1}=\bigoplus_{j=1}^{g_1}\mathbb{C}[\Gamma]e_j\quad \tau_2:\mathbb{C}[\Gamma]^g\rightarrow \mathbb{C}[\Gamma]^{g_2}=\bigoplus_{j=g_1+1}^{g_1+g_2}\mathbb{C}[\Gamma]e_j,\]
where the $e_j$ build the standard $\mathbb{Z}$-basis. We are now ready to prove the first main result of the paper.

\begin{theorem}\label{thm:injective}
Let \(\Gamma\) be a complete intersection numerical semigroup. Under the previous notation, write \(N_\Gamma:=\varphi(\overline{N}_\Gamma)\) and \(N_i:=\varphi_i(\overline{N}_{\Gamma_i})\) for $i=1,2$. Then, the linear maps \(\tau_1,\tau_2\) induce the following injective morphisms 
$$
\Phi_1 : \mathbb{C}[\Gamma_1]^{g_1}/ N_{1} \longrightarrow \mathbb{C}[\Gamma]^{g}/ N_{\Gamma},\quad
\Phi_2 : \mathbb{C}[\Gamma_2]^{g_2}/ N_{2} \longrightarrow \mathbb{C}[\Gamma]^{g}/ N_{\Gamma},
$$
\end{theorem}
\begin{proof}
   First we observe that
   \[
N_\Gamma=\left\langle\left(\sum_{j=1}^g\varphi\Big (\frac{\partial f_j}{\partial u_0}\Big)e_j,\dots, \sum_{j=1}^g\varphi \Big(\frac{\partial f_j}{\partial u_g}\Big)e_j\right)\right\rangle
   \]
   is a submodule of \(\bigoplus_{j=1}^{g}\mathbb{C}[\Gamma]e_j.\)
   Hence, in order to compare \(N_i\) with \(N_\Gamma\) we need first to understand the relation between the maps \(\varphi_i\) and the map \(\varphi.\) Thanks to the tensor product decomposition of \(\mathbb{C}[\Gamma]\) we can write 
   \begin{equation}\label{eqn:paramdecompo}
       \varphi=\pi\circ\lambda_{\rho}\circ({\varphi_1\otimes\varphi_2}),
   \end{equation}
   
   where \(\lambda_\rho\) is the multiplication by \(\rho\) in \(\mathbb{C}[\Gamma_1]\otimes\mathbb{C}[\Gamma_2]\) and \(\pi\) is the canonical projection. By the hypothesis, \(\rho\) is a regular element of \(\mathbb{C}[\Gamma_1]\otimes\mathbb{C}[\Gamma_2]\) and thus yields the exact sequence
   \[0\rightarrow\mathbb{C}[\Gamma_1]\otimes\mathbb{C}[\Gamma_2]\xrightarrow{\lambda_\rho} \mathbb{C}[\Gamma_1]\otimes\mathbb{C}[\Gamma_2]\xrightarrow{\pi}\frac{\mathbb{C}[\Gamma_1]\otimes\mathbb{C}[\Gamma_2]}{(\rho)}\rightarrow 0.\]
   Therefore the following diagram commutes:
   \begin{equation*}
   \begin{tikzcd}[
  column sep=scriptsize,row sep=scriptsize,
  ar symbol/.style = {draw=none,"\textstyle#1" description,sloped},
  isomorphic/.style = {ar symbol={\cong}},
  ]
  \mathbb{C}[u_0,\dots,u_g]\ar[r,"\varphi"]& \mathbb{C}[C^{\Gamma}]\ar[r]&0\\
  \mathbb{C}[x_0,\dots,x_{g_1}]\otimes \mathbb{C}[y_0,\dots,y_{g_2}]\ar[u,isomorphic]\ar[r,"\varphi_1\otimes\varphi_2"]&  \mathbb{C}[\Gamma_1]\otimes\mathbb{C}[\Gamma_2]\ar[u,"\pi\circ \lambda_\rho"]\ar[r]&0.
  \end{tikzcd}
   \end{equation*}
   Now we are ready to prove the statement. We will only check that \(\Phi_1\) is injective, as the result for \(\Phi_2\) may be handled in much the same way.
\medskip

   Let us denote by \(\mathcal{N}:=\tau_1(N_\Gamma),\) the submodule of \(\bigoplus_{j=1}^{g_1}\mathbb{C}[\Gamma]e_j\) which is generated as 
   \[
\mathcal{N}=\left\langle\sum_{j=1}^{g_1}\varphi\Big(\frac{\partial f_j}{\partial u_0}\Big)e_j,\dots, \sum_{j=1}^{g_1}\varphi\Big(\frac{\partial f_j}{\partial u_g}\Big)e_j\right\rangle=\left\langle\sum_{j=1}^{g_1}\varphi\Big(\frac{\partial h_j^1}{\partial x_0}\Big)e_j,\dots, \sum_{j=1}^{g_1}\varphi\Big(\frac{\partial h_j^1}{\partial x_{g_1}}\Big)e_j\right\rangle.
   \]
   The equality follows from the identification \(x_i=u_i\) for \(i=0,\dots,g_1\) and the fact that \(\partial h_j^1/\partial u_j=0\) for \(j=g_1+1,\dots,g.\)
\medskip

The previous considerations yield the commutative diagram
   \begin{equation*}
   \begin{tikzcd}[
  column sep=scriptsize,row sep=normal,
  ar symbol/.style = {draw=none,"\textstyle#1" description,sloped},
  isomorphic/.style = {ar symbol={\cong}},
  ]
  \bigoplus_{j=0}^{g_1}\mathbb{C}[u_0,\dots,u_g]e_j\ar[r,"\varphi_1\otimes\varphi_2"]& \bigoplus_{j=0}^{g_1}\mathbb{C}[C^{\Gamma_1}]\otimes\mathbb{C}[\Gamma_2]e_j\ar[r,"\pi\circ\lambda_\rho"]&\bigoplus_{j=0}^{g_1}\mathbb{C}[C^{\Gamma}]e_j\ar[r,"\pi_1"]&\frac{\bigoplus_{j=0}^{g_1}\mathbb{C}[C^{\Gamma}]e_j}{\mathcal{N}}\ar[r]&0\\
  \bigoplus_{j=0}^{g_1}\mathbb{C}[x_0,\dots,x_{g_1}]e_j\ar[u,hook,"i"]\ar[r,"\varphi_1"]& \bigoplus_{j=0}^{g_1} \mathbb{C}[\Gamma_1]e_j\ar[u,hook,"i"]\ar[r,"\pi_2"]&\frac{\bigoplus_{j=0}^{g_1} \mathbb{C}[\Gamma_1]e_j}{N_1}\ar[ur,dashed,"\Phi'_1"]\ar[r]&0 &
  \end{tikzcd}
   \end{equation*}
   As \(\Phi_1=\tau_1^{-1}\circ\Phi'_1\), for the proof of the injectivity of \(\Phi_1\) it is enough to prove the injectivity of \(\Phi'_1\). In order to show the injectivity of \(\Phi'_1\) it is enough to prove that 
   \[\pi\circ \lambda_\rho\circ i(N_1)=\mathcal{N},\]
   however this follows by definition of \(N_1,\) which is 
   \[
   N_1=\left\langle\sum_{j=1}^{g_1}\varphi_1\Big(\frac{\partial h_j^1}{\partial x_0}\Big)e_j,\dots, \sum_{j=1}^{g_1}\varphi_1\Big(\frac{\partial h_j^1}{\partial x_{g_1}}\Big)e_j\right\rangle,
   \]
   and the fact that the map \(i\) is defined as \(z(x_0\dots,x_{g_1})\mapsto z\otimes 1\in\mathbb{C}[u_0,\dots,u_g].\)
\end{proof}

Observe that Theorem \ref{thm:injective} implies that the base spaces of the miniversal deformations of \(C^{\Gamma_1}\) and \(C^{\Gamma_2}\) are embedded in the miniversal deformation of \(C^{\Gamma}.\) This is quite a specific and intrinsic property of a complete intersection monomial curve, as in this case the splitting of the defining ideal generates a sort of Thom-Sebastiani decomposition of the monomial algebra, in a similar spirit as in \cite{almiron}. In general, if we have a non-monomial complete intersection curve singularity, the defining ideal does not necessarily split in this way, thus this kind of decomposition in separated variables is not available.
\medskip


\subsection{Miniversal deformation of a free semigroup}
In this part, we will consider a particular class of complete intersection numerical semigroups which are called free semigroups. Consider a numerical semigroup \(\Gamma\) generated (not necessarily minimally) by \(G:=\{a_0, a_1, \dots, a_g\}\). Assume that \(G\) satisfies the condition
 \begin{equation}\label{eq:freecond}
     n_ia_i\in\langle a_0, a_1, \dots, a_{i-1} \rangle,
 \end{equation}
for all \(i=1,\dots,g,\) where \(n_{i}:=\gcd(a_0, a_1, \dots, a_{i-1})/\gcd(a_0, a_1, \dots, a_i)\). A numerical semigroup admitting a set of generators \(G\) satisfying \eqref{eq:freecond} for all \(i\geq 1\) was named \textit{free numerical semigroup} by Bertin and Carbonne \cite{beca77}. Moreover, without loss of generality we can further assume that \(n_i>1\) for all \(1\leq i\leq g\). From now on, when referring to a free semigroup \(\Gamma=\langle a_0,\dots, a_g\rangle\) we mean that the set \(G\) satisfies the condition \eqref{eq:freecond} for all \(i\geq 1\) and \(n_i>1\) for all \(1\leq i \leq g.\)
\medskip

Let \(\Gamma=\langle a_0,\dots, a_g\rangle\) be a free semigroup. From the condition \eqref{eq:freecond}, for each \(i\) there exist numbers \(\ell_{0}^{(i)},\dots,\ell_{i-1}^{(i)}\in\mathbb{N}\) such that

\[
n_ia_i=\ell_{0}^{(i)}a_0+\cdots+\ell_{i-1}^{(i)}a_{i-1}.
\]
Therefore it is easy to see that the equations 
\[
f_i=u_{i}^{n_i}-u_{0}^{\ell_{0}^{(i)}}u_{1}^{\ell_{1}^{(i)}}\cdots u_{i-1}^{\ell_{i-1}^{(i)}}=0\quad\text{for}\;1\leq i\leq g
\]
define the curve \(C^{\Gamma}\). In this case, we will explicitly describe the \(\mathbb{C}\)--basis of \(T^1(C^{\Gamma}).\)
To do so, we need first to define an auxiliary set
$$
E_{\ell_0^{(1)},n_1}=\Big \{(i,j)\in \mathbb{N}^2 \ : \ 0\leq i \leq \ell_{0}^{(1)}-2 \ \text{ and } \  0\leq j \leq n_1-2 \Big \}.
$$

In a next step, we need an iteration process to introduce a second auxiliary family of sets as follows: For $s\geq 2$, we consider as base cases

\begin{align*}
    I_{\ell_0^{(s)}}&=\Big\lbrace (i,j)\in \mathbb{N}^2 \ : \ 0\leq i \leq \ell_0^{(s)}-1 \text{ and } \  0\leq j\leq n_1-1\Big\rbrace \\
    I_{\ell_1^{(s)}}&=\Big\lbrace (i,j)\in \mathbb{N}^2 \ : \  0\leq i - \ell_0^{(s)}\leq \ell_0^{(1)}-1 \text{ and } \ 0\leq j\leq \ell_1^{(s)}-1 \Big\rbrace
\end{align*}

so that

\begin{equation*}
    D_{\ell_1^{(s)}}'=\left\{\begin{array}{ll}
       I_{\ell_0^{(s)}}  & \text{if }\ell_1^{(s)} = 0. \\
       I_{\ell_0^{(s)}}\cup I_{\ell_1^{(s)}}  & \text{if }\ell_1^{(s)} \neq 0.
    \end{array}\right.
\end{equation*}
This allows us to define iteratively, for $t=3,\ldots , s$, the sets
\begin{align*}
        I_{\ell_{t-2}^{(s)}}&=\Bigg\lbrace \left(k_0,\ldots, k_{t-1}\right)\in\mathbb{N}^t \ : \ 0\leq k_{t-1}\leq n_{t-1}-1 \text{ and }\left(k_0,\ldots, k_{t-2}\right)\in D_{\ell_{t-2}^{(s)}}'\Bigg\rbrace\\
        I_{\ell_{t-1}^{(s)}}&=\Bigg\lbrace \left(k_0,\ldots, k_{t-1}\right)\in\mathbb{N}^t \ : \ 0\leq k_{t-1}\leq \ell_{t-1}^{(s)}-1 \text{ and }\left(k_0-\ell_0^{(s)},\ldots, k_{t-2}-\ell_{t-2}^{(s)}\right)\in D_{\ell_{t-2}^{(t-1)}}'\Bigg\rbrace\\[2mm]
\end{align*}
and
\begin{equation*}
    D_{\ell_{t-1}^{(s)}}'=\left\{\begin{array}{ll}
       I_{\ell_{t-2}^{(s)}}  & \text{if }\ell_{t-1}^{(s)} = 0, \\
       I_{\ell_{t-2}^{(s)}}\cup I_{l_{t-1}^{(s)}}  & \text{if }\ell_{t-1}^{(s)} \neq 0.
    \end{array}\right.
\end{equation*}

The sets \(D_{\ell_{s-1}^{(s)}}'\) have the following interpretation: 
\begin{lemma}\label{lem:pierrettelemma4}
For every \(s\geq 2\), the set
     \[
     \Big\{u_0^i u_1^j u_2^{k_2}\cdots u_s^{k_s}\;\colon\;(i,j,\ldots , k_{s-1})\in D_{\ell_{s-1}^{(s)}}'\Big\}
     \]
     is a system of generators of \(\mathbb{C}[u_0,\dots,u_s]/(f_1,\dots,f_s).\)
\end{lemma}
\begin{proof}
The statement can be proved in much the same way as \cite[Lemme~4]{pierette}. 
\end{proof}

Finally, let us denote by \(\mathbf{h}(I_\bullet):=\max\{(k_0,\dots,k_t)\in I_\bullet\} \) the maximal point in a set \(I_\bullet\) with respect to the lexicographical order in \(\mathbb{N}^t.\) Thus, for \(s\geq 2\) we define
\[
D_{\ell_{s-1}^{(s)}}:=\left\{\begin{array}{lc}
    D'_{\ell_{s-1}^{(s)}}\setminus\{\mathbf{h}(I_{\ell_{s-2}^{(s)}})\} &\text{if}\;\ell_{s-1}^{(s)}=0,\\[12pt]
    D'_{\ell_{s-1}^{(s)}}\setminus\{\mathbf{h}(I_{\ell_{s-1}^{(s)}})\} &\text{if}\;\ell_{s-1}^{(s)}\neq 0. 
\end{array}\right.
\]

The sets \(D_{\ell_{s-1}^{(s)}}\) allow us to describe a basis of the quotient $\mathbb{C}$-vector space $\mathbb{C}[\Gamma]^g/N$:

\begin{theorem}\label{thm:basisdeffree}
Let \(\Gamma=\langle a_0,\dots, a_g\rangle\) be a free numerical semigroup. Consider the standard basis given by the (column) unit vectors  $\vec{e}_1,\ldots , \vec{e}_g$. A basis of the $\mathbb{C}$-vector space $T^1=\mathbb{C}[\Gamma]^g/N$ consists of the images over $\mathbb{C}[\Gamma]^g/N$ of the following column vectors:
\begin{enumerate}
    \item For $(k_0,k_1)\in E_{\ell_0^{(1)},n_1}$, with $m=2,\ldots , g$ and $k_m=0,1,\ldots , n_m-1$, we have 
    $$(u_0^{k_0}u_1^{k_1}u_2^{k_2}\cdots u_g^{k_g}) \vec{e}_1$$
    \item For $(k_0,k_1)\in D_{\ell_1^{(2)}}$, with $k_2=0,\ldots , n_2-2$, $m=3,\ldots , g$ and $k_m=0,\ldots , n_m-1$, we have
    $$(u_0^{k_0}u_1^{k_1}u_2^{k_2}\cdots u_g^{k_g}) \vec{e}_2$$
    \item For $(k_0,k_1,\ldots , k_{m-1})\in D_{\ell_{m-1}^{(m)}}$ with $k_m=0,\ldots , n_m-2$ and $m'=m+1,\ldots , g$ so that $k_{m'}=0,\ldots , n_{m'}-1$, we have
    $$(u_0^{k_0}u_1^{k_1}u_2^{k_2}\cdots u_g^{k_g}) \vec{e}_m,$$
    for $m=3,\ldots , g-1$.\\[1pt]
    \item For $(k_0,k_1,\ldots , k_{g-1})\in D_{\ell_{g-1}^{(g)}}$ with $k_g=0,\ldots , n_g-2$, we have
    $$(u_0^{k_0}u_1^{k_1}u_2^{k_2}\cdots u_g^{k_g}) \vec{e}_g$$
\end{enumerate}
\end{theorem}
\begin{proof}
    We will proceed by induction on the number of generators of the numerical semigroup \(\Gamma=\langle a_0,\dots,a_g\rangle.\)
\medskip

    Assume first \(g=1.\) In this base case the claim is easy since the monomial curve is the plane curve with equation \(f_1(u_0,u_1)=u_1^{a_0}-u_0^{a_1}.\) Therefore, as \(f_1\) is quasihomogeneous it belongs to the Jacobian ideal \((\partial f_1/\partial u_0,\partial f_1/\partial u_1)=(u_0^{a_1-1},u_1^{a_0-1})\) and then
    \[
    T^1(C^\Gamma)=\frac{\mathbb{C}[u_0,u_1]}{(u_0^{a_1-1},u_1^{a_0-1})}.
    \]
    In this case we denote \(\ell_0^{(1)}=a_1\) and \(\)a \(\mathbb{C}\)--basis of \(T^1(C^\Gamma)\) is the set \(\Big \{u_0^iu_1^j\;\colon\; (i,j)\in E_{\ell_0^{(1)},n_1}\Big\}. \)
\medskip

    Now suppose by induction that the result is true for \(g\leq k\) and let us assume we are in the case \(k=g+1.\) Recall that, since \(\Gamma\) is a free semigroup, we may write \(\Gamma=n_{g+1}\Gamma_g+a_{g+1}\mathbb{N}.\) Also, by Delorme \cite[Proposition~10]{delormeglue} we have \(c(\Gamma)=n_{g+1}c(\Gamma_g)+(n_{g+1}-1)(a_{g+1}-1)\), which according to Pinkham \cite[Sect.~10]{Pinkham1} coincides with the dimension of \(T^1(C^\Gamma)\); observe that $n_{g+1} \neq 1$ since $\Gamma$ is free.
    \medskip
    
    Let us denote by \(\mathcal{B}_g\) the basis of the  \(\mathbb{C}\)-vector space \(\mathbb{C}[\Gamma_g]^g/N_g\) provided by induction hypothesis. Theorem \ref{thm:injective} yields the injective $\mathbb{C}$-linear map
    \[\begin{array}{ccc}
        \mathbb{C}[\Gamma_g]^g/N_g&\xrightarrow{\Phi}&\mathbb{C}[\Gamma_{g+1}]^{g+1}/N\\
        z&\mapsto&(z,0)
    \end{array}\]
    Then, \(\{u_{g+1}^{k}\Phi(z)\;\colon\;z\in\mathcal{B}_g,\;0\leq k\leq n_{g+1}-1\}\) is a set of \(\mathbb{C}\)--linearly independent non-zero elements of  \(\mathbb{C}[\Gamma_{g+1}]^{g+1}/N\) whose cardinal is \(n_{g+1}c(\Gamma_g).\) Moreover, by induction hypothesis those are precisely the set of vectors defined by parts \((1),(2)\) and \((3).\) Therefore, it remains to show that 
    $$
    (u_0^{k_0}u_1^{k_1}u_2^{k_2}\cdots u_g^{k_{g+1}}) \vec{e}_{g+1}
    $$
     for $(k_0,k_1,\ldots , k_{g})\in D_{\ell_{g}^{(g+1)}}$ with $k_g=0,\ldots , n_g-2$ are \((n_{g+1}-1)(a_{g+1}-1)\) non-zero elements which together with the previous vectors build a system of generators of \(\mathbb{C}[\Gamma_{g+1}]^{g+1}/N\).  
     \medskip
     
     By Lemma \ref{lem:pierrettelemma4} we have that \(u_0^{k_0}\cdots u_{g}^{k_{g}}u_{g+1}^{k_{g+1}}\) with \((k_0,\dots,k_{g})\in D'_{\ell_{g}^{(g+1)}}\) is a system of generators of \(\mathbb{C}[u_0,\dots,u_{g+1}]/(f_1,\dots,f_{g+1}).\) Observe that, as each \(f_i\) is homogeneous of degree \(n_ia_i\), our definition of \(D_{\ell_{g}^{(g+1)}}\) from \(D'_{\ell_{g}^{(g+1)}}\) eliminates the unique element of \(D'_{\ell_{g}^{(g+1)}}\) that goes to \(0\) after taking quotient with \(N.\)  Therefore, to conclude the proof we only need to show that the number of elements of \(D'_{\ell_{g}^{(g+1)}}\) is \(a_{g+1}\). To prove that, we recall the definition of the sets $ D_{\ell_{t-1}^{(s)}}'$, namely
     \begin{equation*}
    D_{\ell_{t-1}^{(s)}}'=\left\{\begin{array}{ll}
       I_{\ell_{t-2}^{(s)}}  & \text{if }\ell_{t-1}^{(s)} = 0. \\
       I_{\ell_{t-2}^{(s)}}\cup I_{l_{t-1}^{(s)}}  & \text{if }\ell_{t-1}^{(s)} \neq 0.
    \end{array}\right.
\end{equation*}
for \(s\geq 2\) and \(t=3,\dots,s\). From this definition it is easily seen that  
     \begin{equation}\label{eqn:aux1}
         \Big |D'_{\ell_{t-1}^{(s)}}\Big|=n_{t-1}\Big|D'_{\ell_{t-2}^{(s)}}\Big|+\ell_{t-1}^{(s)}\Big|D'_{\ell_{t-2}^{(t-1)}}\Big|.
     \end{equation}
    
    Recall also that \(\ell_0^{(1)}=a_1/e_1=a_1/(n_2\cdots n_{g+1})\) and then
     \[
|D'_{\ell_{1}^{(g+1)}}|=\ell_0^{(g+1)}n_1+\ell_1^{(g+1)}\ell_0^{(1)}=\frac{\ell_0^{(g+1)}a_0+\ell_1^{(g+1)}a_1}{n_2\cdots n_{g+1}}\quad\text{and}\quad |D'_{\ell_1^{(2)}}|=\frac{a_2}{e_2}=\frac{a_2}{n_3\cdots n_{g+1}}.
     \]
   
Recursively, we can use eqn. \eqref{eqn:aux1} to show 
\begin{equation}
    \label{eqn:cardinalD'}
    \Big|D'_{\ell_{t-2}^{(s)}}\Big|=\frac{\ell_0^{(s)}a_0+\cdots+\ell_{t-2}^{(s)}a_{t-2}}{n_{t-1}\cdots n_{g+1}}\quad\text{and}\quad \Big|D'_{\ell_{t-2}^{(t-1)}}\Big|=\frac{a_{t-1}}{e_{t-1}}=\frac{a_{t-1}}{n_{t-2}\cdots n_{g+1}}
\end{equation}
    Finally, applying the previous computations to the case \(s=g+1\) and \(t=g+1\) and the identity \(n_{g+1}a_{g+1}=\ell_0^{(g+1)}a_0+\cdots+\ell_{g-1}^{(g+1)}a_{g-1}+\ell_{g}^{(g+1)}a_g\) we obtain \(\Big|D'_{\ell_{g}^{(g+1)}}\Big|=a_{g+1}.\) In this way, the number of vectors of the form \((4)\) is \((n_{g+1}-1)(a_{g+1}-1)\), which is the desired conclusion.
\end{proof}

\begin{rem}
    (1)~Theorem \ref{thm:basisdeffree} is a generalization of \cite[Th\'eor\`eme~3]{pierette} in the sense that, if \(\Gamma\) is the semigroup of a plane branch, then our Theorem \ref{thm:basisdeffree} recovers  \cite[Th\'eor\`eme~3]{pierette}.\\
    (2)~ If we allow \(n_i=1\) for some \(i\in \{1,\ldots , g+1\}\), then observe that there is no loss of generality in the proof of Theorem \ref{thm:basisdeffree}. In that case the conductor in the iteration remains constant and the induction is trivial by Theorem \ref{thm:injective}.
\end{rem}

We illustrate Theorem \ref{thm:basisdeffree} by showing an explicit construction of a basis of the $\mathbb{C}$-vector space $\mathbb{C}[\Gamma_g]^g/N$.

\begin{ex}\label{ex:examplebasis}
    Set the sequence of positive integers $A=(18,27,21,32)$. It is easily seen that the sequence of $(n_1,n_2,n_3)$ associated to $S$ is $(2,3,3)$ and that the numerical semigroup $\Gamma_A=\langle A \rangle$ generated by $A$ is free, since 
    \begin{equation}\label{eq:descomp_unica_ejemplo_base_defor}
        \begin{array}{rcl}
        n_1\,a_1 =  2\cdot 27 & = & 3\cdot 18 =\ell_0^{(1)} a_0, \\[2mm]
        n_2\,a_2 =  3\cdot 21 & = & 2\cdot 18 + 1\cdot 27=\ell_0^{(2)} a_0 + \ell_1^{(2)} a_1\\[2mm]
        n_3\,a_3 =  3\cdot 32 & = & 3\cdot 18 + 0\cdot 27 + 2\cdot 21 =\ell_0^{(3)} a_0 + \ell_1^{(3)} a_1 + \ell_2^{(3)} a_2.\\
        \end{array}
    \end{equation}
    Now, let us describe a basis of the quotient $\mathbb C$-vector space $\mathbb{C}[\Gamma_A]^3/N$ taking into account both Theorem \ref{thm:basisdeffree} and eqns.~ \eqref{eq:descomp_unica_ejemplo_base_defor}. First we calculate the elements belonging to the set $E_{\ell_0^{(1)},n_1}$; this is
    $$
        E_{\ell_0^{1},n_1}=\Big \{(i,j)\in \mathbb{N}^2 \ : \ 0\leq i \leq 1 \ \text{ and } \  0\leq j \leq 0 \Big \} =\{(0,0),(1,0)\}. 
    $$
    After this base case, the following two steps are the computation of the elements in the sets $D_{\ell_{1}^{(2)}}$ and $D_{\ell_{2}^{(3)}}.$ But before that, we need to calculate the corresponding sets $I_{\bullet}$ defining them. We start computing the elements in the set $D_{\ell_{1}^{(2)}}$. In this particular case, as $\ell_{1}^{(2)}=1\neq 0,$ we have to calculate $I_{\ell_0^{(2)}},I_{\ell_1^{(2)}},$  $D_{\ell_1^{(2)}}'$ and $\textbf{h}(I_{\ell_1^{(2)}})$:
    \begin{equation*}
        \begin{array}{rl}
        I_{\ell_0^{(2)}}&=\Big\lbrace (i,j)\in \mathbb{N}^2 \ : \ 0\leq i \leq \ell_0^{(2)}-1=1 \text{ and } \  0\leq j\leq n_1-1=1\Big\rbrace\\[3mm] &=\{(0,0),(0,1),(1,0),(1,1)\}. \\[2mm]
        I_{\ell_1^{(2)}}&=\Big\lbrace (i,j)\in \mathbb{N}^2 \ : \  0\leq i - \ell_0^{(2)}\leq \ell_0^{(1)}-1 \text{ and } \ 0\leq j\leq \ell_1^{(2)}-1 \Big\rbrace\\[3mm] &=\{(2,0),(3,0),(4,0)\}.\\[3mm]
        D_{\ell_1^{(2)}}'&=I_{\ell_0^{(2)}}\cup I_{\ell_1^{(2)}} \quad \text{and}\quad \textbf{h}\big(I_{\ell_1^{(2)}}\big)=(4,0).
         
        \end{array}
    \end{equation*}
    This yields the set $D_{\ell_{1}^{(2)}}$, namely 
    $$
        D_{\ell_{1}^{(2)}}= \{(0,0),(0,1),(1,0),(1,1),(2,0),(3,0)\}.   
    $$
    To conclude, we obtain the set $D_{\ell_{2}^{(3)}}.$ In this case,  firstly, we have to compute $I_{\ell_0^{(3)}}$ and $D_{\ell_1^{(3)}}'$ since $\ell_1^{(3)}=0$ and after the sets $I_{\ell_1^{(3)}},I_{\ell_2^{(3)}},$ and $D_{\ell_2^{(3)}}',$ because $\ell_2^{(3)}\neq 0,$ and the vector $\textbf{h}(I_{\ell_2^{(3)}})$:
    \begin{equation*}
        \begin{array}{rl}
        I_{\ell_0^{(3)}}&=\Big\lbrace (i,j)\in \mathbb{N}^2 \ : \ 0\leq i \leq \ell_0^{(3)}-1=2 \text{ and } \  0\leq j\leq n_1-1=1\Big\rbrace\\[3mm] &=\{(0,0),(0,1),(1,0),(1,1),(2,0),(2,1)\}. \\[2mm]
        D_{\ell_1^{(3)}}'&=I_{\ell_0^{(3)}}.
        \\[3mm] 
        I_{\ell_1^{(3)}}&=\Bigg\lbrace \left(k_0,k_1,k_{2}\right)\in\mathbb{N}^3 \ : \ 0\leq k_{2}\leq n_{2}-1=2 \text{ and } \left(k_0, k_{1}\right)\in D_{\ell_{1}^{(3)}}'\Bigg\rbrace\\[5mm]
        &=\{(0,0,0),(0,0,1),(0,0,2),(0,1,0),(0,1,1),(0,1,2),(1,0,0),(1,0,1),\\ & \hspace{6.2mm}(1,0,2),(1,1,0),(1,1,1),(1,1,2),(2,0,0),(2,0,1),(2,0,2),(2,1,0),\\ & \hspace{6.2mm}(2,1,1),(2,1,2)\}. \\[2mm]
        I_{\ell_2^{(3)}}&=\Bigg\lbrace \left(k_0,k_1, k_{2}\right)\in\mathbb{N}^3 \ : \ 0\leq k_{2}\leq \ell_{2}^{(3)}-1=1 \text{ and }\left(k_0-\ell_0^{(3)}, k_{1}-\ell_{1}^{(3)}\right)\in D_{\ell_{1}^{(2)}}'\Bigg\rbrace\\[5mm]
        &=\{(3,0,0),(3,0,1),(3,1,0),(3,1,1),(4,0,0),(4,0,1),(4,1,0),(4,1,1),(5,0,0),\\ & \hspace{6.2mm}(5,0,1),(6,0,0),(6,0,1),(7,0,0),(7,0,1)\}.\\[2mm]
        D_{\ell_2^{(3)}}'&=I_{\ell_1^{(3)}}\cup I_{\ell_2^{(3)}} \quad \text{and} \quad \textbf{h}\big(I_{\ell_2^{(3)}}\big)=(7,0,1).\\[3mm]
        
        \end{array}
    \end{equation*}
    As a result, we get 
     \begin{equation*}
        \begin{array}{rl}
        D_{\ell_{2}^{(3)}}&=\{(0,0,0),(0,0,1),(0,0,2),(0,1,0),(0,1,1),(0,1,2),(1,0,0),(1,0,1),\\ 
        & \hspace{6.2mm} (1,0,2),(1,1,0),(1,1,1),(1,1,2),(2,0,0),(2,0,1),(2,0,2),(2,1,0),\\ 
        & \hspace{6.2mm}(2,1,1),(2,1,2),(3,0,0),(3,0,1),(3,1,0),(3,1,1),(4,0,0),(4,0,1),\\ 
        & \hspace{6.2mm}(4,1,0),(4,1,1),(5,0,0),(5,0,1),(6,0,0),(6,0,1),(7,0,0)\},\\
        \end{array}
    \end{equation*}
    and following Theorem \ref{thm:basisdeffree}, we obtain a basis of the quotient $\mathbb C$-vector space $\mathbb{C}[\Gamma_A]^3/N$. 
\end{ex}


\section{On the dimension of the moduli space of a monomial curve associated to a free semigroup} \label{sec:dimensionmoduli}

Let $B=\mathbb{C}[\Gamma]$ be the monomial ring associated to a numerical semigroup \(\Gamma\) endowed with the usual grading given by $\Gamma$. Recall that the affine monomial curve \(C:=\operatorname{Spec} B\) has a defining ideal \(I=(f_1,\dots,f_k)\) where \(f_1,\dots,f_k\) are homogeneous elements, see e.g.~Herzog~\cite{herzog}. Therefore, there is a natural action of the multiplicative group \(\mathbb{G}_m\) of units. As \(C\) is a curve with isolated singularity, in virtue of Grauert \cite{grauert} there is a miniversal deformation whose tangent space \(T^1_B\) possesses a natural \(\mathbb{Z}\)--grading (see also Greuel et al.~\cite{Greuelbook}). (Recall here that $T^1_B$ can be identified to the Zariski tangent space to the base of the miniiversal deformation of $(C^{\Gamma},\mathbf{0})$ thanks to the fact that the Kodaira-Spencer map is a bijection in this case, see e.g.~\cite[Lemma~1.20]{Greuelbook}).

\subsection{An account of the complete intersection case} We will focus on the case of a complete intersection numerical semigroup \(\Gamma=\langle a_0,\dots,a_g\rangle\). In this case, as mentioned in Section \ref{sec:miversalCI}, Tjurina's theorem \cite{Tjurina} shows that the tangent space of the deformation is 
\[
T^1_B=\frac{\mathbb{C}[u_0,\dots,u_g]^g}{\left(\frac{\partial f_i}{\partial u_j}\right)_{i,j}\mathbb{C}[u_0,\dots,u_g]^{g+1}+(f_1,\dots,f_g)\mathbb{C}[u_0,\dots,u_g]^g}.
\]

Consider a basis \(s_1\dots,s_\tau\in\mathbb{C}[u_0,\dots,u_g]^{g}\) of \(T^1_B\), where \(s_i=(s^1_i,\dots,s^g_i)\) for $i=1,\ldots , \tau$. Then the miniversal deformation of \(C\) can be described as follows: For $k>0$, we define
\begin{equation}\label{eqn:eqdeformation}
    \begin{array}{cc}
	F_1(\mathbf{u},\mathbf{w})=&f_1(\mathbf{u})+\sum_{j=1}^{\tau}w_js^{1}_j(\mathbf{u}),\\
	\vdots&\vdots\\
F_k(\mathbf{u},\mathbf{w})=&f_k(\mathbf{u})+\sum_{j=1}^{\tau}w_js^{k}_j(\mathbf{u})
\end{array}
\end{equation}

and let \((\mathcal{X},\mathbf{0}):=V(F_1,\dots,F_k)\subset(\mathbb{C}^{g+1}\times\mathbb{C}^t,\mathbf{0})\) be the zero set of \(F_1,\dots,F_k\), then the deformation defined by  \((C^{\Gamma},\mathbf{0})\xrightarrow{i}(\mathcal{X},\mathbf{0})\xrightarrow{\phi}(\mathbb{C}^\tau,\mathbf{0})\) is the miniversal deformation of \((C^\Gamma,\mathbf{0})\), where \(i\) is induced by the inclusion and \(\phi\) by the natural projection. In fact, if one chooses \(\deg(w_j)=-\deg(s_j)\), then we endow the algebra \(\mathbb{C}[u_0,\dots,u_g,w_1,\dots,w_\tau]\) with the unique grading for which \(\deg(u_i)=a_i\) and the \(F_i\) are homogeneous with \(\deg(F_i)=\deg(f_i)\). Under this grading, we obtain a partition of the base space \(\mathbb{C}^\tau\) into two parts. Define the sets
\begin{align*}
    P_{+}:=&\{j\in\{1,\dots,\tau\}:\;\deg(w_j)<0\}\\
    P_{-}:=&\{j\in\{1,\dots,\tau\}:\;\deg(w_j)>0\}.
\end{align*}
\begin{rem}
    It is worth noting that---when dealing with a deformation---we have the parameter space with coordinates \(w_1,\dots,w_\tau\) on the one hand, and the basis of \(T^1\) given by Theorem \ref{thm:basisdeffree} on the other hand. Therefore, from the choice of the grading, a parameter with negative weight provides a monomial with positive grading in \(T^1.\) This is the reason that motivates the definition of \(P_{+}\) and \(P_{-}\) as we will use them to refer to positive/negative weight deformations. Observe that the way we have defined these sets is then the opposite of the one given by Teissier in \cite{teissierappen}.
\end{rem}

Denote by $\tau_{+}(\Gamma):=\tau_{+}:=|P_{+}|$ and $\tau_{-}(\Gamma):=\tau_{-}:=|P_{-}|$. Since there are no \(w_j\) of degree zero, we have the equality \(\tau=\tau_{+}+\tau_{-}\). Moreover, there is a natural action of the group \(\mathbb{C}^{\ast}\) over \((\mathcal{X},\mathbf{0})\) which is compatible with the previous construction and that induces the natural action on \(\phi^{-1}(0)\cong C^{\Gamma}\).
\medskip

At this point, according to Pinkham \cite{Pinkham}, in order to study the moduli space associated to \(\Gamma\) we have to focus on the negative part of the deformation, i.e. \(P_{-}\). To this purpose we need first to consider the base change in the deformation induced by the inclusion map defined as \(V_-:=(\mathbb{C}^{\tau_-}\times\{\mathbf{0}\},0)\hookrightarrow (\mathbb{C}^{\tau},\mathbf{0})\), on account of the diagram
\[
	\xymatrix{
		(C^{\Gamma},\mathbf{0})\ar[dr]\ar[r]&(\mathcal{X},\mathbf{0})\ar[rr]& &(\mathbb{C}^{\tau},\mathbf{0})  \\
		& (\mathcal{X}_\Gamma,\mathbf{0}):=(\mathcal{X},\mathbf{0})\times_{(\mathbb{C}^{\tau},\mathbf{0})}(V_-,\mathbf{0})\ar[rr]\ar[u]& &(V_-,\mathbf{0}).\ar[u]}
	\]
 
 Let us denote by \(G_\Gamma: \mathcal{X}_\Gamma\rightarrow V_-\) the deformation induced by this base change. Observe that this deformation can be described in terms of the eqns.~\eqref{eqn:eqdeformation} by making \(w_j=0\) for all \(j\in P_+.\) This is now a negatively graded deformation. Following Pinkham, we must projectivize the fibers of \(G_\Gamma\) without projectivizing the base space \(V_-.\) This can be done by replacing \(s_j\) with \(s_j(u_0,\dots,u_g)X_{g+1}^{-\deg s_j}.\) Observe that we have the inclusion \(\overline{\mathcal{X}}_\Gamma\subset \mathbb{P}^{g+2}\times V_-\) where the ring \(\mathbb{C}[u_0,\dots,u_g,X_{g+1}]\) has \(\deg u_i=a_i\) and \(\deg X_{g+1}=1.\) According to Pinkham \cite[Proposition 13.4, Remark 10.6]{Pinkham1} the morphism
$$
\pi: \overline{\mathcal{X}}_\Gamma\longrightarrow V_{-}
$$
is flat and proper, and has fibres which are reduced projective curves, and all those fibres which are above a given $\mathbb{G}_m$-orbit of $V_{-}$ are isomorphic.
\medskip

This leads Pinkham \cite[Theorem 13.9]{Pinkham1} to prove the following. (Our formulation sticks to Buchweitz \cite[Theorem 3.3.4]{Buchweitz}).

\begin{theorem}[Pinkham]\label{thm:Pinkham}
Let $\mathcal{M}_{g,1}$ be the coarse moduli space of smooth projective curves $C$ of genus $g$ with a section i.e. of pointed compact Riemann surfaces of genus $g$. Let $\Gamma$ be a numerical semigroup and set the subscheme of $\mathcal{M}_{g,1}$ parameterizing pairs
$$
W_{\Gamma} = \Big  \{ (X_0,p) : X_0 \ \mbox{is a smooth projective curve of genus} \ g, \ \mbox{and} \ p\in X_0 \ \mbox{with} \  \Gamma_p=\Gamma  \Big \},
$$
where $\Gamma_p$ is the Weierstrass semigroup at the point $p$. Moreover, write $V^{-}_s$ for the open subset of $V^{-}$ given by the points $u\in V'$ such that the fibre of $\overline{\mathcal{X}}_\Gamma \to V_{-}$ above $u$ is smooth. This is $\mathbb{G}_m$ equivariant, and so there exists a bijection between $W_\Gamma$ and the orbit space $V^{-}_s / \mathbb{G}_m$.
\end{theorem}

As complete intersections can be deformed without obstructions, the following corollary is an easy consequence of Theorem \ref{thm:Pinkham} and Deligne-Greuel's formula \cite{Deligne,Greuel82Del}:
\begin{corollary}\label{cor:dimModuli1}
    Let \(\Gamma\) be a complete intersection numerical semigroup. Then, 
    \[\dim W_\Gamma=\tau_{-}=c(\Gamma)-\tau_{+}.\]
\end{corollary}

\subsection{On the recursive computation of the dimension of the moduli space of a free semigroup} Starting with the recursive presentation of a free semigroup, our aim is to compute \(\tau_+\), and then the dimension of the moduli space, in a recursive way. We will focus on the particular case of free numerical semigroups \(\Gamma_m=\langle a_1,\dots,a_m\rangle\). We will provide sharp upper and lower bounds for the dimension of the moduli space of \(\Gamma_m\) in terms of the dimension of the moduli space of \(\Gamma_{m-1}.\)
\medskip

Before continuing with the procedure to compute the dimension of the moduli space of a free semigroup, we need the following technical result.

\begin{proposition}\label{prop:Proposition1}
Let $\langle a_0,a_1, \ldots , a_g \rangle$ be a free numerical semigroup and 
$$
A_{s,k}=\Big \{(i,j)\in E_{\ell_0^{(1)},n_1} : ia_0+ja_1+ka_s > n_1 a_1 \ \mathrm{and} \ a_s<n_1a_1\Big\},
$$ for $2\leq s\leq g$ and $1\leq k\leq n_s-1$. Set $b_{s,k}:=|A_{s,k}|$.  Then
$$
b_{s,k}=\tau^{+}_{\langle n_1,\ell_0^{(1)} \rangle} + \Big \lfloor \frac{ka_s}{e_1} \Big \rfloor - \sigma_{1,k}(a_s)-\gamma_{1,k}(a_s)+1,
$$
where
\begin{equation*}
    \sigma_{1,k}(t)=\left\{\begin{array}{ll}
       0,  & \text{if } \Big \lfloor \frac{kt}{e_1} \Big \rfloor < \ell_0^{(1)}, \\[10pt]
       1,  & \text{otherwise }.
    \end{array}\right.
\end{equation*}
and
\begin{equation*}
    \gamma_{1,k}(t)=\left\{\begin{array}{ll}
       0,  & \text{if } \Big \lfloor \frac{kt}{e_1} \Big \rfloor < n_1- \Big \lfloor \frac{n_1}{\ell_0^{(1)}} \Big \rfloor\ell_0^{(1)}, \\[10pt]
       \Big \lfloor \frac{n_1}{\ell_0^{(1)}} \Big \rfloor +1, & \text{if } \Big \lfloor \frac{kt}{e_1} \Big \rfloor \geq n_1,\\[10pt]
       \Big \lfloor \frac{kt}{e_1} \Big \rfloor -n_1+\Big \lfloor \frac{n_1}{\ell_0^{(1)}}\Big \rfloor\ell_0^{(1)},  & \text{if }  n_1-\Big \lfloor \frac{n_1}{\ell_0^{(1)}} \Big \rfloor \ell_0^{(1)} \leq \Big \lfloor\frac{kt}{e_1} \Big \rfloor  \leq n_1.
    \end{array}\right.
\end{equation*}
\end{proposition}

\begin{proof}
First of all, observe that we have a partition of \(E_{\ell^{(1)}_0,n_1}\) as 
\[E_{\ell^{(1)}_0,n_1}=\left\{(i,j)\in E_{\ell^{(1)}_0,n_1} \ \colon \   ia_0+ja_1>n_1a_1 
\right\}\sqcup\left\{(i,j)\in E_{\ell^{(1)}_0,n_1} \ \colon \   ia_0+ja_1<n_1a_1 
\right\}.
\]

As by definition, \(\Big \{(i,j)\in E_{\ell^{(1)}_0,n_1} \ \colon \   ia_0+ja_1>n_1a_1 \Big \}\subset A_{s,k}\) then at least there are  $\tau^{+}_{\langle n_1,\ell_0^{(1)} \rangle} $ points in \(A_{s,k}.\)
\medskip

Set the index $s\in \{2,3,\ldots , g\}$ such that $a_s<n_1a_1$ and $1\leq k \leq n_s-1$. Now, assume $(i,j)\in E_{\ell_0^{(1)},n_1}$ such that $ia_0+ja_1<n_1a_1$; we want to see how many of those points satisfy $ia_0 + j a_1+k a_2>n_1a_1$ for some \(k>0\). Since $ia_0+ja_1<n_1a_1$, we have that $ia_0+ja_1=n_1a_1-\varepsilon e_1$ if and only if $i n_1 + j\ell_0^{(1)}=n_1\ell_0^{(1)}-\varepsilon$. We know that, in 
$$
B=\{(i,j) : 0\leq i \leq \ell_0^{(1)}-1, \ 0 \leq j \leq n_1 -1\},
$$ 
the line $i n_1 + j \ell_0^{(1)}=n_1\ell_0^{(1)}-\varepsilon$ contains a unique point. Assume for a moment that $(i,j)\in B$, then
$$
ia_0+ja_1+ka_2-n_1a_1 > 0 \ \Longleftrightarrow \ n_1a_1 - \varepsilon e_1 + ka_2 - n_1a_1>0,
$$

which is equivalent to $0\leq \varepsilon \leq \Big \lfloor \frac{k a_2}{e_1} \Big \rfloor$. This means that there are, at most, 
$$
\Big \lfloor \frac{k a_s}{e_1} \Big \rfloor +1
$$
points such that $(i,j)\in E_{\ell_0^{(1)},n_1}$ with $ia_0+ja_1<n_1a_1$ and $ia_0 + j a_1+k a_2>n_1a_1$ for some \(k>0\). At this moment, the two following situations may appear:
\medskip

\noindent Case 1: $\varepsilon =n_1+j\ell_0^{(1)}$ with $0\leq j \leq n_1-1$ and $(\ell_0^{(1)})n_1+j\ell_0^{(1)}=n_1\ell_0^{(1)}-\varepsilon$.
\medskip

\noindent Case 2: $\varepsilon =\ell_0^{(1)}-in_1$ with $in_1 + n_1\ell_0^{(1)}-\ell_0^{(1)}=n_1\ell_0^{(1)}-\varepsilon$.
\medskip

Without loss of generality we can assume that \(a_0>a_1\). Since $\langle a_0, a_1, \ldots , a_g \rangle$ is a free semigroup, we have $n_1=a_0/n_1$ and $\ell_0^{(1)}=a_1/e_1$, which implies $n_1>\ell_0^{(1)}$. Thus, in Case 2, as $\varepsilon \geq 0$, the only possibility is $i=0$. So, under the hypothesis of $\Big \lfloor \frac{k a_s}{e_1}\Big \rfloor < \ell_{0}^{(1)}$, there will be no points satisfying the conditions of Case 2, and under the hypothesis of $\Big \lfloor \frac{k \delta_s}{e_1}\Big \rfloor \geq  \ell_{0}^{(1)}$, there will be only one point satisfying the conditions of Case 2. This justifies the definition of $\sigma_{1,k}(a_s)$. On the other hand, the points satisfying the conditions of Case 1 can be studied through the function $\gamma_{1,k}(a_s)$, so that we obtain 
$$
b_{s,k}=\tau^{+}_{\langle n_1,\ell_0^{(1)} \rangle} + \Big \lfloor \frac{ka_s}{e_1} \Big \rfloor +1 - \sigma_{1,k}(a_s)-\gamma_{1,k}(a_s),
$$
as wished.
\end{proof}

Now, we can proceed with the main result of this section which relates \(\tau_m^+\) and \(\tau_{m-1}^{+}\). If \(\Gamma=\langle a_0,\dots, a_g\rangle\) is a free semigroup, then its monomial curve is defined by \(f_1,\dots,f_g\in\mathbb{C}[u_0,\dots,u_g]\) equations of degrees \(\deg(f_i)=n_ia_i\). According to the discussion at the beginning of this section, we can endow \(T^1\) with a grading in such a way \(\deg(u_i)=a_i\) and the equations of the deformation \(F_i(\mathbf{u,w})=f_i(\mathbf{u})+\sum_{j=1}^{\tau}w_js^1_j(\mathbf{u})\) are homogeneous with \(\deg(F_i)=\deg(f_i)\), for $i=1,\ldots , k$. In this way, the monomial basis of \(T^1\) provided by Theorem \ref{thm:basisdeffree} have the following weights

\begin{enumerate}
    \item For $(k_0,k_1)\in E_{\ell_0^{(1)},n_1}$, with $r=2,\ldots , m$ and $k_r=0,1,\ldots , n_r-1$, we have 
    $$\deg\Big((u_0^{k_0},u_1^{k_1},u_2^{k_2}\cdots u_m^{k_m}) \vec{e}_1\Big)=\sum_{i=1}^{m}k_ia_i-n_1a_1.$$
    \item For $(k_0,k_1)\in D_{\ell_1^{(2)}}$, with $k_2=0,\ldots , n_2-1$, $r=3,\ldots , m$ and $k_r=0,\ldots , n_r-1$, we have
    $$\deg\Big((u_0^{k_0},u_1^{k_1},u_2^{k_2}\cdots u_m^{k_m}) \vec{e}_2\Big)=\sum_{i=1}^{m}k_ia_i-n_2a_2.$$
    \item For $(k_0,k_1,\ldots , k_{r-1})\in D_{\ell_{r-1}^{(r)}}$ with $k_r=0,\ldots , n_r-1$ and $r'=r+1,\ldots , m$ so that $k_{r'}=0,\ldots , n_{r'}-1$, we have
    $$\deg\Big((u_0^{k_0},u_1^{k_1},u_2^{k_2}\cdots u_m^{k_m}) \vec{e}_r\Big)=\sum_{i=1}^{m}k_ia_i-n_ra_r,$$
    for $r=3,\ldots , m-1$.\\[1pt]
    \item For $(k_0,k_1\ldots , k_{m-1})\in D_{\ell_{m-1}^{(m)}}$ with $k_m=0,\ldots , n_m-2$, we have
    $$\deg\Big((u_0^{k_0},u_1^{k_1},u_2^{k_2}\cdots u_m^{k_m}) \vec{e}_m\Big)=\sum_{i=1}^{m} k_ia_i-n_ma_m$$
\end{enumerate}
Set $d_{m,k}= | D_{\ell_{m-1}^{(m)},k}^{+}|$ where \(
    D_{\ell_{m-1,k}^{(m)}}^{+}= \Big \{(k_0,\ldots, k_{m-1})\in D_{\ell_{m-1}^{(m)}} : \displaystyle \sum_{i=1}^{m-1} k_ia_i > (n_m-k) a_m \Big \}\) for \(k=0,\dots,n_m-2.\) We have therefore the following.

\begin{theorem}\label{thm:dimensionpositive}
Let \(\Gamma_m\) be a free numerical semigroup with \(m\) generators. Then,

\[\tau_{m-1}^{+} + (n_m-1)(\mu_{m-1} +d_{m,0})+\sum_{k=1}^{n_m-2}\left\lfloor\frac{ka_m}{n_m}\right\rfloor \geq \tau_m^{+}\geq \tau^{+}_{m-1}+(n_m-1)(\tau^{+}_{m-1}+d_{m,0}).\]

Moreover, 
\[
\tau_{m}^{+}\geq \tau^{+}_{\langle n_1,\ell_{0}^{(1)}\rangle}+\sum_{\substack{j\notin L_1\\ j\notin J_m}}
(n_j-1)\tau^{+}_{\langle n_1,\ell_{0}^{(1)}\rangle}+\sum_{j\in L_1} \Big(\sum_{k=1}^{n_j-1} b_{j,k}\Big)+(n_m-1)\Big(\sum_{j\in J_m}\frac{a_j}{e_j}+d_{m,0} \Big ),
\]

where \(J_m:=\big\{j\in\{1,\dots,m-1\} : a_m\geq n_ja_j\big\},\) $L_1:=\big\{i \in \{2,\ldots , m\} : a_i < n_1 a_1\big\}$ and \(b_{j,k}\) is defined as in Proposition \ref{prop:Proposition1}.
\end{theorem}

\begin{proof}
As in the proof of Theorem \ref{thm:basisdeffree}, let us denote by \(\mathcal{B}_{m-1}\) the \(\mathbb{C}\)--basis of \(\mathbb{C}[\Gamma_{m-1}]^{m-1}/N_{m-1}.\) By Theorem \ref{thm:injective} we have the injective map
    \[\begin{array}{ccc}
        \mathbb{C}[\Gamma_{m-1}]^{m-1}/N_{m-1}&\xrightarrow{\Phi}&\mathbb{C}[\Gamma_{m}]^{m}/N_m\\
        \overline{z}&\mapsto&(\overline{z},0)
    \end{array}\]
    Then, Theorem \ref{thm:basisdeffree} shows that the basis \(\mathcal{B}_m\) decomposes as
    \[\mathcal{B}_m=\Big \{u_{m}^{k}\Phi(\overline{z})\;\colon\;\overline{z}\in\mathcal{B}_{m-1},\;0\leq k\leq n_{m}-1\Big\}\sqcup\Big\{(u_0^{k_0}u_1^{k_1}u_2^{k_2}\cdots u_m^{k_m}) \vec{e}_m\;\colon\;(k_0,k_1\ldots , k_{m-1})\in D_{\ell_{m-1}^{(m)}}\Big\}.\]

Let us first assume that \(a_m\geq a_{j}n_{j}\) for all \(j=1,\dots,m-1\), then for all \(k\geq 1\) and denote by \(z=\Phi(\overline{z})\) for \(\overline{z}\in\mathcal{B}_{m-1}.\) We have
\[\deg(u^{k}_mz)=\deg(z)+ka_m=\sum_{i=0}^{m-1}\alpha_ia_i-n_ja_j+ ka_m\geq 0.\]
Therefore, if \(a_m\geq a_{j}n_{j}\) for all \(j=1,\dots,m-1\), then the decomposition of the basis implies that
\begin{equation}\label{eqn:dimension1}
    \tau^{+}_m=\tau_{m-1}^{+} + (n_m-1)\mu_{m-1} + \sum_{k=0}^{n_m-2} d_{m,k}.
\end{equation}

Now, let us assume the existence of an index \(j_0\in\{1,\dots,m-1\}\) such that \(a_m<a_{j_0}n_{j_0}.\) Thus, there exists \(z\in\mathcal{B}_{m-1}\) such that \(\deg(u_mz)<0.\) Hence, in this case we obtain the strict inequality
\begin{equation}\label{eqn:dimension2}
    \tau^{+}_m<\tau_{m-1}^{+} + (n_m-1)\mu_{m-1} + \sum_{k=0}^{n_m-2} d_{m,k}.
\end{equation}

Let us now move to provide a lower bound. Set \(J_m=\big\{j\in \{1,\ldots, m-1\} : a_m \geq n_ja_j\big\}\) and assume \(\{1,\dots,m-1\}\setminus J_m\neq \emptyset\) as otherwise we are in the previous situation. Observe that, by the decomposition of the basis \(\mathcal{B}_m,\) any element \(z\in\mathcal{B}_{m-1}\) with \(\deg(z)\geq 0\) also satisfies \(\deg(a_m^kz)\geq 0\) for \(k=0,\dots,n_{m}-1.\) This implies the inequality \(\tau_m^{+}\geq n_{m}\tau_{m-1}^{+}.\) Analogously, for any \(j\in J_m\) and any \(z\in D_{\ell^{(j)}_{j-1}}\) such that \(\deg(z)\leq 0\) we have \(\deg(a_m^kz)\geq 0\) for \(k=1,\dots,n_{m}-1.\) Thus, if we set \(d^{-}_{j}=\Big \{z\in D_{\ell^{(j)}_{j-1}}\ |\ \deg(z)<0\Big\} \) for \(j\in J_m\), then
\begin{equation}\label{eqn:dimension3}
    \tau_m^{+}\geq n_{m}\tau_{m-1}^{+}+(n_m-1)(\sum_{j\in J_m}d^{-}_{j})+\sum_{k=0}^{n_m-2}d_{m,k}.
\end{equation}

Independently of the assumptions on \(a_m\geq n_ja_j\) or \(a_m<n_{j_0}a_{j_0}\), let us now estimate the sum \(\displaystyle\sum_{k=0}^{n_m-2}d_{m,k}.\) Recall that $d_{m,0}=\Big|D_{\ell_{m-1,0}^{(m)}}^{+}\Big|$ and $d^{'}_{m,0}=\Big|D_{\ell_{m-1}^{(m)}}^{'+}\Big|=d_{m,0}+1$, where

\begin{align*}
   D_{\ell_{m-1,0}^{(m)}}^{'+}=&\Big \{(k_0,\ldots, k_{m-1})\in D^{'}_{\ell_{m-1}^{(m)}} : \sum_{i=1}^{m-1} k_ia_i > n_m a_m \Big \}\\
    D_{\ell_{m-1,k}^{(m)}}^{'+}=& \Big \{(k_0,\ldots, k_{m-1})\in D^{'}_{\ell_{m-1}^{(m)}} : \sum_{i=1}^{m-1} k_ia_i > (n_m-k) a_m \Big \}
\end{align*}
This implies that 
$$
0\leq \Big|D^{'+}_{\ell_{m-1,k}^{(m)}} \setminus D^{'+}_{\ell_{m-1,0}^{(m)}}\Big|\leq \Big \lfloor \frac{ka_m}{n_m} \Big \rfloor,
$$ 
therefore $d^{'}_{m,0}\leq d^{'}_{m,k}=d_{m,k}+1 \leq d^{'}_{m,0}+\lfloor k a_m/n_m \rfloor$. We deduce then that
\begin{equation}
 \label{eqn:dimension4}   
 (n_m-1)(d^{'}_{m,0}-1)=\sum_{k=0}^{n_m-2} (d^{'}_{m,0}-1) \leq \sum_{k=0}^{n_m-2}d_{m,k}
\leq (n_m-1) (d^{'}_{m,0}-1)+\left\lfloor\frac{k a_m}{n_m} \right\rfloor.
\end{equation}

As \(d^{-}_j\geq 0\) for all \(j\in J_m\), then a combination of \eqref{eqn:dimension1},\eqref{eqn:dimension2},\eqref{eqn:dimension3} and \eqref{eqn:dimension4} provides the desired inequalities
\[\tau_{m-1}^{+} + (n_m-1)(\mu_{m-1} +d_{m,0})+\sum_{k=1}^{n_m-2}\left\lfloor\frac{ka_m}{n_m}\right\rfloor \geq \tau_m^{+}\geq \tau^{+}_{m-1}+(n_m-1)(\tau^{+}_{m-1}+d_{m,0}).\]
\medskip
To finish, let us show the inequality 
\[\tau_{m}^{+}\geq \tau^{+}_{\langle n_1,\ell_{0}^{(1)}\rangle}+\sum_{\substack{j\notin L_1\\ j\notin J_m}}(n_j-1)\tau^{+}_{\langle n_1,\ell_{0}^{(1)}\rangle}+\sum_{j\in L_1} (\sum_{k=1}^{n_j-1} b_{j,k})+(n_m-1)\big(\sum_{j\in J_m}\frac{a_j}{e_j}+d_{m,0}\big).\]
Observe that the basis \(\mathcal{B}_{m-1}\) is computed through the sets \(E_{\ell_0^{(1)},n_1}\) and \(D_{\ell_{s-1}^{(s)}}.\) 

In this way, it is obvious that in \(\mathcal{B}_m\) there are at least \(\tau^{+}_{\langle n_1,\ell_{0}^{(1)}\rangle}\) positive weight elements which are precisely those of the form \((u_0^iu_1^j)\vec{e}_1\) with \((i,j)\in E_{\ell_0^{(1)},n_1}\) such that \(ia_0+ja_1>n_1a_1.\) Now, under the notations of Proposition \ref{prop:Proposition1}, there are $\displaystyle\sum_{j\in L_1} \displaystyle\sum_{k=1}^{n_j-1} b_{j,k}$ positive weight elements which are of the form \((u_0^iu_1^ju_s^k)\vec{e}_1\) such that \(ia_0+ja_1+ka_s>n_1a_1\) with \((i,j)\in E_{\ell_0^{(1)},n_1}\) and \(a_s<n_1a_1.\) 

Additionally, for any \(r\in J_m\) we have \(\deg\big((u_0^{k_0}u_1^{k_1}u_2^{k_2}\cdots u_m^{k_m}) \vec{e}_r\big)>0\) if $(k_0,k_1,\ldots , k_{r-1})\in D_{\ell_{r-1}^{(r)}}$ with $k_r=0,\ldots , n_r-2$ and $r'=r+1,\ldots , m-1$ so that $k_{r'}=0,\ldots , n_{r'}-1$ and \(k_m=1,\dots,n_m-1\). Observe that the number of such elements is precisely \((n_m-1)\big|D_{\ell_{r-1}^{(r)}}\big|.\) Recall that eqn.~\eqref{eqn:cardinalD'} implies \(\Big|D_{\ell_{r-1}^{(r)}}\Big|=a_r/e_r.\) Therefore, the previous discussion provides the desired inequality.
\end{proof}

An immediate consequence of Theorem \ref{thm:dimensionpositive} is the following:
\begin{corollary}
    Let \(\Gamma=\langle a_1,\dots,a_m\rangle\) be a free numerical semigroup. Then the dimension \(\tau^{-}_m\) of the moduli space \(W_\Gamma\) satisfies the following inequalities
    
   \[n_m\tau_{m-1}^{-}-(n_m-1)(d_{m,0}+a_m-1)\geq \tau_m^{-}\geq \tau_{m-1}^{-}-(n_m-1)(d_{m,0}+a_m-1)-\sum_{k=1}^{n_m-2}\left\lfloor \frac{ka_m}{n_m}\right \rfloor.\]
\end{corollary}

The sets \(D_{\ell_{j-1}^{(j)}}\) used to determine the monomial basis of \(T^1\) define recursively a \(j\)-dimensional staircase in \(\mathbb{N}^m.\) Observe that in order to determine \(d_{m,k}\) or \(d^{-}_j\) we need to count how many points of this staircase are below and over the hyperplane defined by \(n_ja_j=\ell_0^{j}a_0+\cdots+\ell_{j-1}^{j}a_j.\) Without any extra assumptions on the generators of \(\Gamma_m,\) it is quite a difficult task from a combinatorial point of view to provide an exact formula for \(d_{m,k}\) or \(d^{-}_j\) and hence for \(\tau_m^{+}.\) However, we can be more precise if we impose some extra conditions over the generators of the semigroup. In fact, we will show that the bounds of Theorem \ref{thm:dimensionpositive} are sharp.
\medskip

Let us first start with the following new characterization the semigroup of values of an irreducible plane curve:

 \begin{theorem}\label{prop:charirreduciblesemigroup}
     Let \(\Gamma=\langle a_0,\dots,a_m\rangle\) be a free numerical semigroup. Then the following statements are equivalent:
     \begin{enumerate}
         \item \(n_ia_i<a_{i+1}\) for all \(i,\) i.e. \(\Gamma\) is the semigroup of values of an irreducible plane curve singularity.
         \item The dimension of the positive part of \(T^1\) can be computed recursively via the formula
         \[\tau^{+}_m=\tau_{m-1}^{+} + (n_m-1)\mu_{m-1} + \sum_{k=0}^{n_m-2} d_{m,k}.\]
     \end{enumerate}
       Moreover for a free semigroup satisfying the conditions $(1)$ we have 
         \[\sum_{k=0}^{n_m-2} d_{m,k}=\frac{(n_m-1)\mu_{m-1}}{2}+\frac{(n_m-3)(a_m/e_m-3)}{2}+\left\lfloor\frac{a_m}{e_mn_m}\right\rfloor-2\]
 \end{theorem}
 \begin{proof}
 $(1)\Leftrightarrow(2)$ is a consequence of the first part of the proof of Theorem \ref{thm:dimensionpositive} as inductively equations \eqref{eqn:dimension1} and \eqref{eqn:dimension2} shows that the upper bound in Theorem \ref{thm:dimensionpositive} is attained if and only if \(n_ia_i<a_{i+1}\) for all \(i\). The formula for \(\sum_{k=0}^{n_m-2} d_{m,k}\) is obtained by \cite[Th\'eor\`eme~8]{pierette}.  
 \end{proof}

Before to show that the lower bound is also sharp, let us first prove the following

\begin{proposition}\label{prop:lemma32}
Let \(\Gamma_m=\langle a_1,\dots,a_m \rangle\) be a free numerical semigroup with \(m\) generators. Under the previous notation, we have
$$
\frac{(\ell_0^{(1)}-1)(n_1-2)}{2} + \left\lfloor \frac{kq_2}{n_2}\right\rfloor - 1 \geq \Big |D^{+}_{\ell_1^{(2)},k}\Big | \geq \frac{(\ell_0^{(1)}-1)(n_1-2)}{2}-|C| -1,
$$
where \(C\) is the triangle defined by the lines $i=0$, $j=n_1$ and $in_1 + j \ell_0^{(2)}=n_1\ell_0^{(2)}$. In particular, if \(n_2a_2>n_1a_1\) then the upper bound is attained. 
\end{proposition}
\begin{proof}
Let us first denote \(q_2=a_2/e_2.\) We have \(n_2a_2=\ell^{(2)}_0a_0+\ell^{(2)}_1a_1.\)

Assume $\ell_1^{(2)}=0$, then $D'_{\ell_1^{(2)}}=I_{\ell_0^{(2)}}$; consider
$$
A=\{(i,j)\in I_{\ell_0^{(2)}} : i n_1+j\ell_0^{(1)}>q_2\}.
$$
In order to count the points in $A$, we must distinguish two situations:
\medskip

\noindent Case 1: $n_2 a_2 > n_1 a_1$.
\medskip

The inequality $n_2 a_2 > n_1 a_1$ is equivalent to $\ell_0^{(1)}n_1<q_2$. The lines
\begin{align*}
r \equiv i n_1 + j \ell_0^{(1)}& = n_1\ell_0^{(1)}\\
s \equiv i n_1 + j \ell_0^{(2)}& = q_2
\end{align*}
are parallel. The lines $r$, $i=\ell_0^{(1)}$ and $j=n_1$ enclose a triangle $B$, and $s$, $i=\ell_0^{(2)}$ and $j=n_1$ the triangle $A$. We have that 
$$
|A| = |B| = \frac{(\ell_0^{(1)}-1)(n_1 - 1)}{2}=\frac{c(\Gamma_1)}{2}.
$$

	\begin{center}
		\begin{figure}[H]
			\begin{tikzpicture}[scale=0.85]
			
			\draw[draw=none, pattern=north west  lines, pattern color=blue] (0,3) -- (6, 3) -- (6,0);
			\draw[draw=none, pattern=north east lines, pattern color=red] (2.7,3) -- (8.5, 3) -- (8.5,0);
			\draw[ultra thick] (0,5) -- (0,0) -- (9,0);

            \draw[] (0,4.5) -- (8.5,0) -- (9,-0.25);
             \draw[] (0,3) -- (6,0) -- (6.5,-0.25) ;
   
			\draw[] (6,3) -- (6,-0.5);
			\draw[] (3,3) -- (-0.5,3);

           \draw[] (3,3) -- (8.5,3) -- (8.5,-0.5);
           
			
			\node [below left][DimGray] at (-0.3,3.3) {$n_1$};
			\node [below][DimGray] at (6,-0.5) {$\ell_0^{(1)}$};
   \node [below][DimGray] at (6.7,-0.1) {$r$};
			\node [below ][DimGray] at (8.7,-0.50) {$\ell_0^{(2)}$};
   \node [below][DimGray] at (9.2,0) {$s$};
			\node at (7.5,2) {$|A|$};
			\node at (3.6,1.8) {$|B|$};
			\end{tikzpicture}	
			
			\caption{Case $\ell_1^{(2)}=0$ and $n_2 a_2 > n_1 a_1$.} \label{fig:lattice1}
		\end{figure}
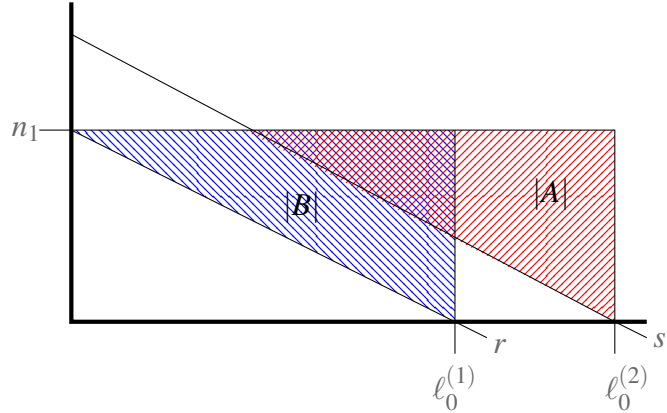
	\end{center}

Case 2: $n_2 a_2 < n_1 a_1$.
\medskip

Consider again the parallel lines 
\begin{align*}
r \equiv i n_1 + j \ell_0^{(1)}& = n_1\ell_0^{(1)}\\
s \equiv i n_1 + j \ell_0^{(2)}& = q_2
\end{align*}

Now the inequality $n_2 a_2 < n_1 a_1$ is equivalent to $\ell_0^{(1)}n_1>q_2$ so that a triangle $C$ delimited by the lines $r$, $s$ and $j=n_1$ appears. The interesting area here is that of the quadrilateral $\overline{A}$ (see Figure \ref{fig:32case2}), namely
$$
|\overline{A}| = |A| - |C| = |B| - |C| = \frac{(\ell_0^{(1)}-1)(n_1-2)}{2} - |C|.
$$
In this way we can bound the area of the region $\overline{A}$ both above and below. 
	\begin{center}
		\begin{figure}[H]
			\begin{tikzpicture}[scale=0.85]
			
		\draw[draw=none, pattern=north west  lines, pattern color=blue] (-2.5,3) -- (6, 3) -- (6,0);
			\draw[draw=none, pattern=north east lines, pattern color=red] (0,3) -- (8.5, 3) -- (8.5,0);

\draw[draw=none, pattern=north west  lines, pattern color=yellow] (-2.5,3) -- (0, 3) -- (0,2);
  
			\draw[ultra thick] (0,3.5) -- (0,0) -- (9.2,0);

           \draw[] (8.5,3) -- (9,3);
            \draw[] (0,3) -- (8.5,0) -- (9,-0.2);
            \draw[] (-2.5,3) -- (6,0) -- (6.5,-0.25) ;
   
			\draw[] (6,3) -- (6,-0.5);
			\draw[] (3,3) -- (-2.5,3);

           \draw[] (3,3) -- (8.5,3) -- (8.5,-0.5);
           
			
			\node [below left][DimGray] at (9.9,3.3) {$n_1$};
			\node [below][DimGray] at (5.7,-0.5) {$\ell_0^{(2)}$};
   \node [below][DimGray] at (6.7,-0.1) {$s$};
			\node [below ][DimGray] at (8.7,-0.50) {$\ell_0^{(1)}$};
   \node [below][DimGray] at (9.2,0) {$r$};
			\node at (7,2) {$|B|$};
			\node at (2,1.9) {$|\overline{A}|$};
			\node at (-0.4,2.65) {$|C|$};
			\end{tikzpicture}	
			
			\caption{Case $\ell_1^{(2)}=0$ and $n_2 a_2 < n_1 a_1$.} \label{fig:32case2}
		\end{figure}
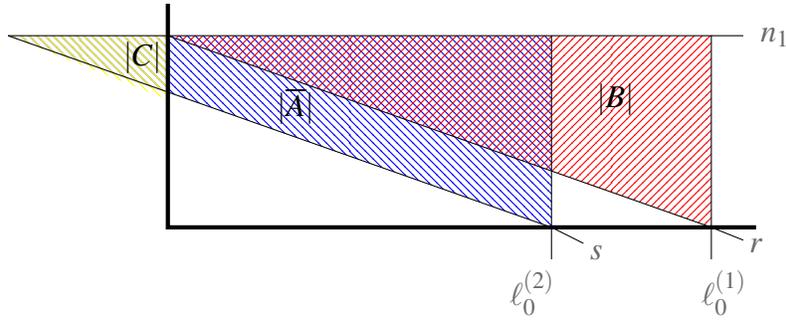
	\end{center}

 Assume now that $\ell_{1}^{(2)}\neq 0$, then $D_{\ell_1^{(2)}}'=I_{\ell_0^{(2)}}\cup I_{\ell_{1}^{(2)}}$, and so the set $A$ can be written as a union $A=A_1\cup A_2$, where

 \begin{align*}
     A_1=&\{(i,j)\in I_{\ell_{0}}^{(2)} : i n_1 + j \ell_0^{(2)}>q_2\}\\
     A_2=&\{(i,j)\in I_{\ell_{1}}^{(2)} : i n_1 + j \ell_0^{(1)}>q_2\}.
 \end{align*}

 We must distinguish again two cases:
 \medskip

 \noindent Case 1: $n_2a_2> n_1a_1$.
\medskip

As in the previous case, we have that $|B|=\frac{c(\Gamma_1)}{2}=|A|$.

	\begin{center}
		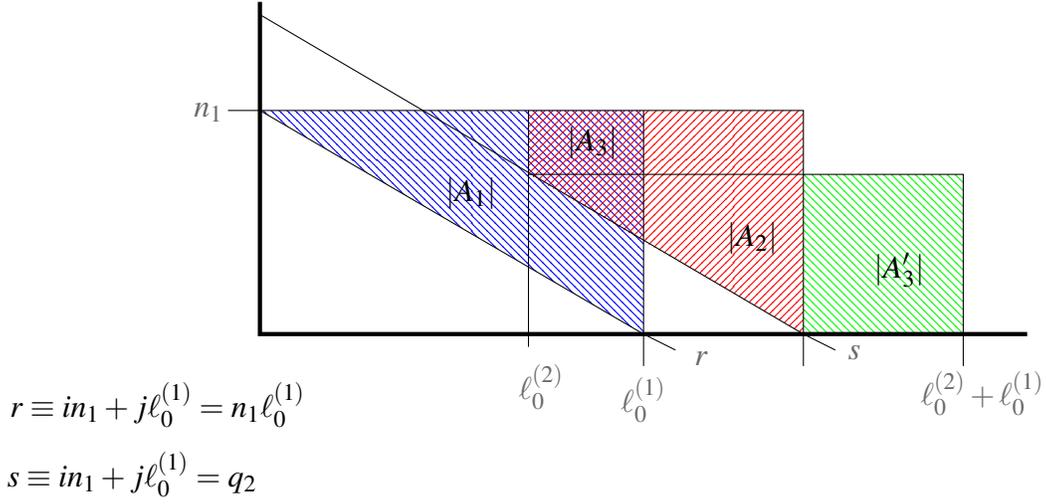
\begin{figure}[H]
			\begin{tikzpicture}[scale=0.85]
			
			\draw[draw=none, pattern=north west  lines, pattern color=blue] (0,3.5) -- (6, 3.5) -- (6,0);
			\draw[draw=none, pattern=north east lines, pattern color=red] (8.5,3.5) -- (8.5, 0) -- (4.2,2.5) -- (4.2,3.5);
			\draw[draw=none, pattern=north west lines, pattern color=green] (8.5,2.5) -- (11, 2.5) -- (11,0) -- (8.5,0);
			\draw[ultra thick] (0,5.2) -- (0,0) -- (12,0);

            \draw[] (0,5) -- (8.5,0) -- (9,-0.25);
             \draw[] (0,3.5) -- (6,0) -- (6.5,-0.25) ;
   
			\draw[] (6,3.5) -- (6,-0.5);
			\draw[] (3.5,3.5) -- (-0.5,3.5);

           \draw[] (3.5,3.5) -- (8.5,3.5) -- (8.5,-0.5);


 \draw[] (4.2,2.5) -- (11,2.5) -- (11,-0.5);

  \draw[] (4.2,3.5) -- (4.2,-0.2) ;
           
			
			\node [below left][DimGray] at (-0.4,3.8) {$n_1$};

   \node [below][DimGray] at (4.4,-0.3) {$\ell_0^{(2)}$};
			\node [below][DimGray] at (6,-0.5) {$\ell_0^{(1)}$};
   \node [below][DimGray] at (6.9,-0.1) {$r$};
			\node [below ][DimGray] at (11.3,-0.4) {$\ell_0^{(2)}+\ell_0^{(1)}$};
   \node [below][DimGray] at (9.3,0) {$s$};
			\node at (7.7,1.5) {$|A_2|$};
			\node at (3.3,2.2) {$|A_1|$};
			\node at (5.2,3) {$|A_3|$};
   \node at (10,1) {$|A_3'|$};
    \node at (-1.6,-1.1) {$r\equiv i n_ 1 + j \ell_0^{(1)}=n_1\ell_0^{(1)} $};
    \node at (-2,-2.2) {$s\equiv i n_ 1 + j \ell_0^{(1)}=q_2 $};
			\end{tikzpicture}	
			
			\caption{Case $\ell_1^{(2)}\neq 0$ and $n_2 a_2 > n_1 a_1$.} \label{fig:lattice3}
		\end{figure}
	\end{center}

We have to check that $|A_3|=|A_3'|$, which is true, as the comparison of the following two easy computations shows:
\begin{align*}
    |A_3|=&(\frac{q_2}{n_1}-\ell_0^{(2)})(n_1-\ell_1^{(2)})=q_2-\frac{\ell_1^{(2)}q_2}{n_1} - \ell_0^{(2)} n_1  \\
    |A_3'|=& \ell_1^{(2)} \left(\ell_0^{(2)} + \ell_0^{(1)} - \frac{q_2}{n_1}\right)
\end{align*}

  \noindent Case 2: $n_2a_2< n_1a_1$.

  	\begin{center}
		\begin{figure}[H]
			\begin{tikzpicture}[scale=0.85]
			
			\draw[draw=none, pattern=north east lines, pattern color=red] (0,3) -- (8.5, 3) -- (8.5,0);

\draw[draw=none, pattern=north west  lines, pattern color=yellow] (-2.5,3) -- (0, 3) -- (0,2);
  
			\draw[ultra thick] (0,3.5) -- (0,0) -- (9.2,0);

           \draw[] (8.5,3) -- (9,3);
            \draw[] (0,3) -- (8.5,0) -- (9,-0.2);
            \draw[] (-2.5,3) -- (6,0) -- (6.5,-0.25) ;
   
			\draw[] (4,3) -- (4,1.6) -- (8.5,1.6);
			\draw[] (3,3) -- (-2.5,3);

           \draw[] (3,3) -- (8.5,3) -- (8.5,-0.5);
           
			\draw[fill] (4,1.6) circle [radius=0.1]; 
			
			\node [below left][DimGray] at (9.9,3.3) {$n_1$};
			\node at (6,2.3) {$|B|$};
			\node at (-0.5,2.65) {$|C|$};
			\end{tikzpicture}	
			
			\caption{Case $\ell_1^{(2)}\neq 0$ and $n_2 a_2 < n_1 a_1$.} \label{fig:lattice2}
		\end{figure}
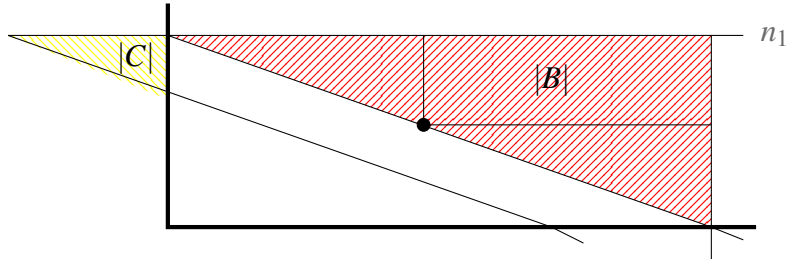
	\end{center}

  As in the previous cases, $|\overline{A}|=|A|-|C|=|B|-|C|=\frac{c(\Gamma_1)}{2}-|C|$, where $C$ is again a triangle. Moreover, as in the case $\ell_1^{(2)}=0$, the part $|C|$ which is loosed comes from $I_{\ell_0^{(2)}}$.

Finally, observe that the points in
$$
A=\{(i,j) \in D'_{\ell_1^{(2)}} : ia_0 + j a_1 + k a_2 > n_2  a_2\}
$$
are those satisfying $in_1 + j \ell_{0}^{(1)} > q_2 - \frac{kq_2}{n_2}$. From the previous pictures, it is easy to see that if \(n_2a_2>n_1a_1\) then there are exactly $\lfloor \frac{kq_2}{n_2} \rfloor$ points \((i,j)\in D^{'}_{l_{1}^{(2)}}\) satisfying 
\begin{equation}\label{eqn:prop3.7aux1}
   q_2 > in_2 + j \ell_0^{(1)} > q_2 - \frac{kq_2}{n_2}. 
\end{equation}
However, if \(n_2a_2<n_1a_1\) it may happen that some of the possible solutions of eqn.~\eqref{eqn:prop3.7aux1} may belong to the triangle \(C\) and thus not in \(D^{'}_{l_{1}^{(2)}}.\) From which we obtain the desired results.

\end{proof}

 Proposition \ref{prop:lemma32} will help us to show that under certain assumption we can give a precise formula for \(\tau_m^{+}\) and in particular to show that the lower bound is sharp.

\begin{theorem}\label{th:moduliatinfinity}
    Let \(\Gamma=\langle a_0,a_1,a_2\rangle\) be a free numerical semigroup such that \(n_1a_1>a_2\) and \(a_0>a_1\) such that \(n_2a_2>n_1a_1.\) Then, 
\[\tau_2^{+}=n_2\tau_1^{+}+\frac{(n_2-1)(c(\Gamma_1)-2)}{2}-\frac{(n_2-1)a_2}{e_1}+\sum_{k=1}^{n_2-1}\left(2\Big \lfloor \frac{ka_2}{e_1} \Big \rfloor - \sigma_{1,k}(a_s)-\gamma_{1,k}(a_s)+1\right).\]
In particular, if \(n_2=2\), then the lower bound in Theorem \ref{thm:dimensionpositive} is attained.
\end{theorem}
\begin{proof}
As in the proof of Theorem \ref{thm:dimensionpositive}, we have the basis of the miniversal deformation described as
 \[
 \mathcal{B}_2=\{u_{2}^{k}\Phi(z)\vec{e}_1\;\colon\;z\in\mathcal{B}_{1},\;0\leq k\leq n_{2}-1\}\sqcup\Big \{(u_0^i u_1^j u_2^{k_2}) \vec{e}_2\;\colon\;(i,j)\in D_{\ell_{1}^{(2)}}\Big \}.
 \]
 Recall that \(\mathcal{B}_1\) is in one to one to correspondence with \(E_{l_0^{(1)},n_1}.\) As by hypothesis we have \(n_1a_1>a_2\) therefore \(\deg(u_{2}^{k}\Phi(u_0^iu_1^j)\vec{e}_1)>0\) if and only if 
 \[
 (i,j)\in A_{2,k}=\Big \{(i,j)\in E_{\ell_0^{(1)},n_1} : ia_0+ja_1+ka_2 > n_1 a_1 \ \mathrm{and} \ a_2<n_1a_1\Big \}.
 \]
 Finally, the number of positively weighted elements in \(D_{l_1^{(2)}}\) are those corresponding to the the sets \(D^{+}_{l_1^{(2)},k}\) with \(k=0,\dots,n_2-2.\) Hence, by using Proposition \ref{prop:lemma32} and Proposition \ref{prop:Proposition1} we obtain
 \[\tau_2^{+}=n_2\tau_1^{+}+\frac{(n_2-1)(c(\Gamma_1)-2)}{2}-\frac{(n_2-1)a_2}{e_1}+\sum_{k=1}^{n_2-1}\left(2\Big \lfloor \frac{ka_2}{e_1} \Big \rfloor - \sigma_{1,k}(a_s)-\gamma_{1,k}(a_s)+1\right),\]
 which completes the proof.
\end{proof}

As one may observe in Proposition \ref{prop:lemma32} and Proposition \ref{prop:Proposition1}, the computation of \(d_{m,k}\) for a free numerical semigroup without any restrictions can be extremely difficult. The reason is that the numbers \(d_{m,k}\) depend on the relations between the different \(n_ja_j\) as we realize looking at the proof of Proposition \ref{prop:lemma32}. Similarly to Proposition \ref{prop:lemma32} one should be able to write out an algorithmic formula or estimation for \(\sum d_{m,k}\). However, the algorithmic calculation is a hard combinatorial problem. Given the length and tedious nature of these combinatorics for more than three generators, it is clearly a good problem to assign to a computer program as Theorem \ref{thm:basisdeffree} provides the explicit basis of the deformation. 

\subsection*{A final remark on the dimension of the moduli space}
 Throughout the section we have been working without loss of generality with a system of generators \(\{a_0,\dots, a_g\}\) of the numerical semigroup satisfying the condition \(n_ia_i\in \langle a_0,\dots,a_{i-1}\rangle\) and \(n_i>1\) for all $i=1,\ldots , g$. However, our statements encode more general information. More concretely, the monomial basis of \(T^1\) is clearly not unique. Nevertheless, the dimension of the moduli space is intrinsic to the semigroup itself, so in particular it does not depend on the generating set of the semigroup. Our results should be then interpreted as follows: given a free numerical semigroup, there exists a system of generators that allows us to explicitly compute a monomial basis for \(T^1.\) This monomial basis leads us to the estimation of the dimension of the moduli space in a recursive way. By looking at Corollary \ref{cor:dimModuli1}, a part of this estimation is independent of the system of generators---essentially the one given by Theorem \ref{thm:injective}--- and the other part actually depends on \(\{a_0,\dots,a_g\}\) as it is linked to the obtained monomial basis.


\section{Moduli of plane curves with fixed semigroup}\label{sec:4}
In this closing section we will focus on the comparison of the two special cases of free numerical semigroups which have already appeared in the previous sections, namely

\begin{enumerate}[(a)]
    \item\label{cond:beta} \(\Gamma\) is free and \(n_{i}a_i<a_{i+1}\) for \(1\leq i\leq g-1\), i.e. plane curve semigroups;
    \item \label{cond:delta} \(\Gamma\) is free and \(n_{i}a_i>a_{i+1}\) for \(1\leq i\leq g-1\) i.e. semigroups of curves with only one place at infinity.
\end{enumerate}


 
According to Bresinsky and Teissier \cite{bresinsky72,teissierappen} given a free semigroup of type \((a)\) there always exists a germ of irreducible plane curve singularity with such a semigroup. Moreover, in contrast to the case of negative homogeneous deformations of Pinkham \cite{Pinkham} (see Section \ref{sec:dimensionmoduli}), Teissier \cite{teissierappen} showed that one can construct the analytic moduli space associated to a branch with such a semigroup by using the positive part of the deformation. This moduli space contrast with the one constructed by Pinkham, as Teissier's moduli refers to the analytic equivalence of germs on \(\mathbb{C}^2.\)

\subsection*{Teissier's and Pinkham's moduli}

Consider an irreducible germ of plane curve singularity in \(\mathbb{C}^2.\) Its topological type is encoded in its semigroup of values \(\langle \overline{\beta}_0,\dots,\overline{\beta}_g\rangle\) which it is well known to be a free semigroup satisfying the condition \(n_i\overline{\beta}_i<\overline{\beta}_{i+1}.\) Moreover, one can see that the generators of this semigroup can be associated with a family of approximates of the curve.
\medskip

Concerning plane curve semigroups, Teissier shows the following. 
\begin{theorem}\cite[I 2.10]{teissierappen}
	The deformation of \(C^{\Gamma}\), obtained by applying the base change \((D_{\Gamma},\mathbf{0}):=(\mathbb{C}^{\tau_{+}}\times\{\mathbf{0}\},\mathbf{0})\hookrightarrow(\mathbb{C}^{\tau},\mathbf{0})\) to the miniversal deformation of \(C^{\Gamma}\), is a miniversal constant semigroup deformation of \(C^{\Gamma}\), i.e. we have the following commutative diagram
	
	\[
	\xymatrix{
		(C^{\Gamma},\mathbf{0})\ar[dr]\ar[r]&(\mathcal{X},\mathbf{0})\ar[rr]& &(\mathbb{C}^{\tau},\mathbf{0})  \\
		& (\mathcal{X}_\Gamma,\mathbf{0}):=(\mathcal{X},\mathbf{0})\times_{(\mathbb{C}^{\tau},\mathbf{0})}(D_{\Gamma},\mathbf{0})\ar[rr]\ar[u]& &(D_{\Gamma},\mathbf{0}).\ar[u]}
	\]
\end{theorem}

As in the case of plane curves constant semigroup is equivalent to constant topological type, it is now natural to consider the analytic moduli space of plane curves with a given semigroup. Analytically equivalence of germs induces an equivalence relation \( \thicksim \) in \( D_\Gamma \). The topological space \( D_\Gamma / \thicksim \), with the quotient topology, will be denoted by \( \widetilde{M}_\Gamma \) and it is called the \emph{moduli space} associated to the semigroup \( \Gamma \). Let \( m : D_\Gamma \longrightarrow \widetilde{M}_\Gamma \) be the natural projection and let \( D^{(2)}_\Gamma \) be the following subset of \( D_\Gamma \)
	\[ D^{(2)}_\Gamma := \{ \boldsymbol{v} \in D_\Gamma \ |\ (G^{-1}(\boldsymbol{v}), \boldsymbol{0})\ \textrm{is a plane branch} \}.  \]
	Then, Teissier proves in \cite[Chap. II, 2.3 (2)]{teissierappen} that \( D^{(2)}_\Gamma \) is an analytic open dense subset of \( D_\Gamma \) and that \( m(D^{(2)}_\Gamma) \) is the analytic moduli space \( M_\Gamma \) of plane branches with semigroup \( \Gamma \) in the sense of Zariski \cite{Zarbook}.~
 Observe that this moduli constrasts with the moduli that we could compute from the point of view of Pinkham given in Section \ref{sec:dimensionmoduli}. In that case, we used the negative part in the deformation and in this case we are using the positive part of the deformation. 
\medskip

On the other hand, free semigroups of type \((b)\) appear in the context of global geometry of plane projective curves with only one place at infinity. Recall that if we denote by $L$ the line at infinity in the compactification of the affine plane to $\mathbb{P}^2$. It is said that the curve $C$ \emph{has only one place at infinity} if the intersection $C\cap L$ is a single point $p$ and $C$ has only one analytic branch at $p$. The point $p$ is often called \emph{the point at the infinity}. For these curves, we can define the following semigroup, which is nothing but the Weiertra\ss~semigroup at the point \(p\):
\begin{defin}
    Let $C$ be a  curve having only one place at infinity, with $p$ the point at infinity. We call \emph{semigroup at infinity} of $C$ to the subsemigroup  $S_{C,\infty}:=\left\{-\nu_{C,p}(h) \ | \ h\in A\right\}$  of $\mathbb N$, where $A$ is the $k$-algebra $\mathcal{O}_C(C\setminus\{p\})$.
\end{defin}
It is a very well known theorem by Abhyankar and Moh \cite{AMoh} that such semigroup is finitely  generated. See the references by Abhyankar \cite{Abh3}, Abhyankar and Moh \cite{AMoh,AbhMoh1}, Sathaye and Stenerson \cite{SatSte} and Suzuki \cite{suzuki} for further information. In fact, Abhyankar-Moh theorem \cite{AMoh} shows that the minimal system of generators \(\{\delta_0,\dots,\delta_s\}\) of $S_{C,\infty}$ satisfies:
\begin{itemize}
\item[(1)] $e_{s}=1$ and $n_i>1$ for every $i \in \{1,\ldots,s\}$, where $n_0:=1$ and $n_i:=e_{i-1}/e_{i}$ for $i\in\{1,\ldots, s\}$ with $e_i:=\gcd (\delta_0,\delta_1,\ldots,\delta_{i})$ for $i\in\{0,1,\ldots, s\}$.

\item[(2)] For $i\in \{1,\ldots,s\}$, the integer $n_i\delta_i$ belongs to the semigroup $\mathbb{N}\delta_0+\mathbb{N}\delta_1+\cdots + \mathbb{N}\delta_{i-1}$.

\item[(3)] $\delta_i< n_{i-1} \delta_{i-1} $
for every $i\in \{1,\ldots,s\}$.
\end{itemize}

In particular, those conditions imply that \(S_{C,\infty}\) is a free semigroup of type \(b).\) It is usual to call \(\delta\)--sequence to a set of generators $\Delta$ of the semigroup at infinity; here we write $\langle \Delta \rangle$. Since we focus on the point at infinity of a plane projective curve, there are two different affine charts depending on whether we consider the line at infinity or its complement. 
\medskip

Set $(X:Y:Z)$ the homogeneous coordinates on $\mathbb{P}^2$ and consider that $Z=0$ is the equation of the line at infinity $L$ and $(1:0:0)$ the coordinates of the point at infinity $p$. Suppose that $(x,y)$ are the coordinates of the affine chart given by the condition $Z\neq 0$ and $(u=y/x,v=1/x)$ the coordinates around the point $p$. Notice that, given a polynomial $f(x,y)$ in $ k[x,y]$, it holds that 
$$
    f(x,y)=v^{-\deg(f)}\hat{f}(u,v),
$$
where $\hat{f}(u,v)$ is the expression of $f(x,y)$ in the local coordinates $(u,v)$. For simplicity, we denote by $q$ an element of $k[x,y],$ or $k[u,v],$ and by $\nu_{C,p}(q)$ the value $\nu_{C,p}(\hat{q}),$ where $\hat{q}$ is the element in the fraction field of $\mathcal{O}_{C,p}$ corresponding to $q.$ Therefore, one can show \cite{GalMonJalg} that there is a relation between the generators of $S_{C,\infty}$ and the generators of the semigroup \(\langle\overline{\beta}_0,\dots,\overline{\beta}_g\rangle\) of irreducible plane curve at its (unique) singular point (which is precisely the point at infinity):

\begin{enumerate}[label=(\textbf{R.\arabic*})]
        \item\label{R1} If $\nu_{C,p}(u)$ does not divide $\nu_{C,p}(v)$, then $s=g, \overline{\beta}_0=\delta_0-\delta_1,\overline{\beta}_1=\delta_0$ and $\overline{\beta}_i=\frac{\delta_0^2}{e_{i-1}}-\delta_i$ for $1<i\leq g$.
        \item\label{R2} If $\nu_{C,p}(u)$ does divide $\nu_{C,p}(v)$, then $s=g+1,\overline{\beta}_0=\delta_0-\delta_1$  and $\overline{\beta}_{i-1}=\frac{\delta_0^2}{e_{i-1}}-\delta_i$ for $1< i \leq g+1$.
    \end{enumerate}

At this point, we may observe the following relation between the monomial curves defined by both semigroups. Thanks to the relations \ref{R1}, we can naturally define a well-defined graded morphism between graded algebras:
\[\begin{array}{ccc}
     \mathbb{C}[Y_0,\dots,Y_g]&\xrightarrow{\Phi}& \mathbb{C}[u_0,\dots,u_g] \\
     Y_0&\mapsto&u_0:=\frac{Y_1}{Y_0}\\
     Y_1&\mapsto&u_1:=\frac{1}{Y_0}\\
    Y_i&\mapsto&u_i:=\frac{Y_i}{Y_0^{\overline{\beta}_1/\overline{e}_{i-1}}}.
\end{array}
\]
Observe that from this morphism, \(\mathbb{C}[u_0,\dots,u_g]\) is naturally graded by \((-\overline{\beta}_0,\dots,-\overline{\beta}_g).\) Obviously, one can perform an isomorphism of graded algebras in order to have a positive the positive grading \((\overline{\beta}_0,\dots,\overline{\beta}_g)\) in \(\mathbb{C}[u_0,\dots,u_g].\) Therefore, it is easy to see that any negatively graded deformation of the monomial curve associated to a \(\delta\)--sequence give rise to a positively graded deformation of the monomial curve associated to a \(\overline{\beta}\)--sequence. A particular case of such deformations is the defining equations of the curves.
\medskip 

As the numerical semigroup generated by a $\delta$-sequence is free, there exists a unique expression of $n_i\delta_i$ for all $i=1,\ldots,n$ such that 
\[n_i\delta_i=\ell_{0}^{(i)}\delta_0+\cdots+\ell_{i-1}^{(i)}\delta_{i-1},\]
where \(\ell_{0}^{(i)},\dots,\ell_{i-1}^{(i)}\in\mathbb{N}\) such that $\ell_{0}^{(i)}\geq 0$ and $\ell_j^{(i)}<n_j,$ for $j=1,\ldots,i-1$.
Therefore, we have that 
\begin{equation}\label{eq_monomialcurve_freesemigroup}
f_i=Y_{i}^{n_i}-Y_{0}^{\ell_{0}^{(i)}}Y_{1}^{\ell_{1}^{(i)}}\cdots Y_{i-1}^{\ell_{i-1}^{(i)}}=0\quad\text{for}\;1\leq i\leq g,
\end{equation}
where \(Y_0,\dots,Y_g\) are  affine coordinates with weight $w(Y_i)=\delta_i$. We can now consider the following fiber \(C_{\infty}\in\mathbb{C}^{g+1}\) of the miniversal deformation of the monomial curve

\begin{equation}\label{equation_monomialY_g}
\left\lbrace \begin{array}{rll}
    h_1:= &f_1(Y_0,Y_1)+tY_2=0 & \\
    h_i:= & f_i(Y_0,\dots,Y_i)+tY_{i+1}=0 & \text{for}\; i=2,\dots,g-1, \\
    h_g:=& f_g(Y_0,\dots,Y_g)=0.&
\end{array}\right.
\end{equation}

It follows easily that \(C_{\infty}\) is a plane curve embedded in \(\mathbb{C}^{g+1}.\) As this is a negative weight deformation, following Pinkham \cite{Pinkham1} the projectivization of the fibers gives rise to a plane curve with only one place at infinity. On the other hand, we can consider the affine chart such that \(u_0=Y_1/Y_0, u_1=1/Y_0\). The relation \ref{R1} yields
$$
w(u_0)=-(w(Y_1)-w(Y_0))=\overline{\beta}_0 \text{ and } w(u_1)=-(-w(Y_0))=\overline{\beta}_1
$$
Now, we rewrite eqns. \eqref{equation_monomialY_g} in the new affine coordinates. To do so, consider first
\[
    h_1(Y_0,Y_1,Y_2)=Y_1^{n_1}-Y_0^{\ell_0^{(1)}} -tY_2=0.
\]
In order to see the last equation in the affine chart $X_0\neq 0$, we must divide $h_1(Y_0,Y_1,Y_2)$ by $Y_0^{\delta_1/e_1}$. The relation \ref{R1} leads to
$$
    \dfrac{\delta_1}{e_1}=\dfrac{\overline{\beta}_1-\overline{\beta_0}}{\overline{e}_1} \ \ \ \ \text{ and } \ \ \ \ \overline{n}_1=\dfrac{\overline{\beta}_0}{\overline{e}_1},
$$
and therefore
\[ 
\dfrac{Y_1^{n_1}-Y_0^{\delta_1/e_1}-tY_2}{Y_0^{\delta_1/e_1}}  =Y_0^{\overline{n}_1}\bigg(u_0^{\overline{\beta}_1/\overline{e}_1}-u_1^{\overline{n}_1}-t\frac{Y_2}{Y_0^{\overline{\beta}_1/\overline{e}_1}}\bigg).
\]

The weight of the affine coordinate \(u_2\colon =\frac{Y_2}{Y_0^{\overline{\beta}_1/\overline{e}_1}}\) is
$$
w(u_2)=-\bigg(w(Y_2)- \frac{\overline{\beta}_1}{\overline{e}_1} w(Y_0)\bigg) =\overline{\beta}_2
$$
by \ref{R1}. In general, for $i=2,\ldots , g-1$ we have
$$
    \dfrac{n_i\delta_i}{e_0}=\dfrac{n_i}{\delta_0}\bigg(\dfrac{\delta_0^2}{e_{i-1}}-\overline{\beta}_i\bigg)=\dfrac{\overline{\beta}_1}{\overline{e}_i}-\dfrac{\overline{n}_i\overline{\beta}_i}{\overline{\beta}_1}
$$
and, from the equation $h_i=0$ in \eqref{equation_monomialY_g}, for $i=2,\ldots,g-1$ it follows that 
$$
\dfrac{f_i(Y_0,\dots,Y_i)+tY_{i+1}=0}{Y_0^{n_i\delta_i/e_0}}=Y_0^{\overline{n}_i\overline{\beta}_i/\overline{\beta}_1}\bigg(h(u_0,\ldots,u_i)-t\dfrac{Y_{i+1}}{Y_0^{\overline{\beta}_1/e_i}}\bigg) 
$$
Consequently, the weight of the coordinate \(u_{i+1}\colon =\frac{Y_{i+1}}{Y_0^{\overline{\beta}_1/\overline{e}_i}}\) is 
$$
w(u_{i+1})=-\bigg(w(Y_{i+1})-\frac{\overline{\beta}_1}{\overline{e}_i} w(Y_0)\bigg) =\overline{\beta}_{i+1}.
$$
The case $\nu_{C,p}(u)$ does divide $\nu_{C,p}(v)$ follows in a similar manner, by using the relation \ref{R2} instead.
\medskip

Observe that the change of charts give rise to an irreducible plane curve singularity with semigroup \(\langle\overline{\beta}_0,\dots,\overline{\beta}_g\rangle.\) According to the discussion at the begining of the section, those equations belong to a certain class in the moduli space of branches with semigroup \(\Gamma.\) 

\subsection*{Relation between the moduli of a plane branch semigroup and the semigroup at infinity}
Therefore, it is natural to ask for possible links between the geometry of the moduli space of branches with fixed semigroup \(\Gamma\) à la Teissier and the geometry of the moduli space à la Pinkham of projective curves with only one point at infinity with fixed Weierstra\ss~semigroup \(\langle \Delta \rangle\). The following examples show the difficulty of this question.

\begin{ex}
    Let us consider the semigroup of an irreducible germ of plane curve defined by \(\langle 6,15,77\rangle.\) Again in virtue of the relations \ref{R1} and \ref{R2} we can see that there is no \(\delta\)--sequence that can be associated to this semigroup. 
\end{ex}

As we have seen, given a \(\delta\)--sequence associated to a semigroup of a plane branch \(\Gamma\), the equation at \(0\) is a plane equisingular deformation of the monomial curve associated to \(\Gamma\). Thus, it is natural to ask if the different fibers associated to a \(\delta\)--sequence provide different analytic types of plane branches with fixed semigroup \(\Gamma\). Let us first show a family of curves for which this actually happens.

\begin{ex}\label{example:Zariskiinvariant}
    
Let $C:f(x,y)=0$ be the germ of an irreducible plane curve singularity with one Puiseux pair, i.e. a germ of plane curve with equation 
$$
f(x,y)=x^{\overline{\beta}_1}-y^{\overline{\beta}_0}+\sum_{i\overline{\beta}_0+j\overline{\beta}_1>\overline{\beta}_0\overline{\beta}_1}a_{i,j}x^{i}y^{j},
$$
    in an adequate choice of coordinates and where $\gcd(\overline{\beta}_1,\overline{\beta}_0)=1$. An important analytic invariant is the Zariski exponent associated to $C$. The Zariski exponent was defined by Zariski in \cite[p.~25]{Zarbook} as follows. If $\nu(\Omega)$ stands for the valuation of the differential form $\Omega=\overline{\beta}_1 y dx - \overline{\beta}_0 x dy$, then the Zariski exponent of $C$ is defined to be 
\[
\lambda=\lambda(C) = \nu(\Omega)- n + 1.
\]

For the precise meaning of valuation of a differential form, the reader is refereed again to Zariski \cite[p.~25]{Zarbook}.
\medskip

In the case of a curve with one Puiseux pair, Peraire \cite{peraire} showed that one can identify the Zariski exponent with a unique monomial in the previous equation. In fact, it is the \((i,j)\) such that \(\overline{\beta_1}~\overline{\beta}_0+(\lambda-\overline{\beta}_1)=i\overline{\beta}_0+j\overline{\beta}_1.\) Thanks to this property, whenever possible, we can produce a family of \(\delta\)-sequences which produces topologically equivalent curves with different Zariski exponent. 
\medskip

To the semigroup \(\Gamma\) we may always associate a family of \(\delta\)--sequences defined by 
\[\delta_0=a,\quad\delta_1=a-1,\delta_2=a^2\overline{\beta}_0-\overline{\beta_1},\quad\text{where}\quad \frac{\overline{\beta}_1}{\overline{\beta}_0}>a>\sqrt{\frac{\overline{\beta}_1}{\overline{\beta}_0}}.\]

Obviously, as in the previous example, it may happen that such an \(a\) does not exist. In that case, there is a unique \(\delta\)--sequence associated to \(\Gamma\) which is \(\delta_0=\overline{\beta}_1\) and \(\delta_1=\overline{\beta}_1-\overline{\beta}_0\). Observe that in that case, the associated equation is precisely the quasi-homogeneous term \(y^{\overline{\beta}_0}-x^{\overline{\beta}_1}\).
\medskip

If there exists such a family, the monomial curve associated to each \(\delta\)--sequence is
\[f_1=u_1^a-u_0^{a-1},\quad f_2=f_1^{\overline{\beta}_0}-u_{1}^{\ell_1}u_0^{\ell_0},\]
where \(\ell_1(a-1)+\ell_0a=\delta_2.\) The germ \(C_0\) at \(\mathbf{0}\) is 
\[
(u_1^a-u_0)^{\overline{\beta}_0}-u_1^{\ell_1}u_0^{a\overline{\beta}_0-(\ell_1+\ell_0)}=0.
\]

Under the analytic change of coordinates \(x=u_1,y=u_1^a-u_0\), \(C_0\) is analytically equivalent to \[f(x,y)=y^{\overline{\beta}_0}-x^{\overline{\beta}_1}-\sum_{k=1}^{a\overline{\beta}_0-(\ell_1+\ell_0)}{\alpha \choose k} (-1)^kx^{\overline{\beta}_1-ak}y^k,\]
where \(\alpha=a\overline{\beta}_0-(\ell_1+\ell_0).\) Therefore, it is straightforward to see that the Zariski exponent of \(f\) only depends on \(a\); it is thus different for every $a$.

\end{ex}

\begin{ex}

Let \(\Gamma=\langle 10,36,183\rangle\) be a numerical semigroup. 
Consider the following $\delta$-sequences such that the semigroup of values of the corresponding curves with only one place at infinity is $\Gamma$ (in fact they are the only ones with this prescribed $\Gamma$):
$$
\Delta_1=\langle 36,26,465\rangle, \ \ \Delta_2=\langle 20,10,4,17\rangle, \ \ \Delta_3=\langle 30,20,54,267\rangle.
$$

 Consider the monomial curve $C_{\Delta_1}$ parametrized by the generators in the delta sequence $\Delta_1$. This curve satisfies the equations

\begin{equation*}
   C_{\Delta_1}:\left\lbrace \begin{array}{rl}
      f_1=   & y^{18} - x^{13} \\
      f_2=   & f_1^2- x^{15}y^{15}
    \end{array}\right.
\end{equation*}

The second equation therein will be transformed into

\begin{equation*}
    \begin{array}{rl}
         \hat{f}_1=&v^{-36}\left(\left(u^{18}-v^5\right)^2-u^{15}v^6\right) 
    \end{array}
\end{equation*}

when we consider affine coordinates in the $(u,v)$-chart. Similarly, for the equations of the curve parametrized according to $\Delta_2$, the third equation in

\begin{equation*}
    C_{\Delta_2}:\left\lbrace  \begin{array}{rl}
        f_1= & y^2 -x \\
        f_2= & f_1^5-x\\
        f_3= & f_2^2 -xyf_1 = \left(\left(y^2-x\right)^5-x\right)^2- xy\left(y^2-x\right)\\
    \end{array}\right.
    \end{equation*}
  reads off in the affine $(u,v)$-coordinates,
after the changes of variables $(w=u,s=u^2-v)$ and $(w=u,v=-s+w^2)$, as

\begin{equation*}
    \begin{array}{rl}
        \hat{f}_2& =
                 \left(s^5 - \left(-s+w^2\right)^9\right)^2 -sw\left(-s+w^2\right)^{16}.
    \end{array}
\end{equation*}

Last we consider the monomial curve $C_{\Delta_3}$ and the equations

\begin{equation*}
    C_{\Delta_3}:\left\lbrace  \begin{array}{rl}
        f_1= & y^3 -x^2 \\
        f_2= & f_1^5-x^{9}\\
        f_3= & f_2^2 - x^{16}f_1 =  \left(\left(y^3-x^2\right)^5-x^9\right)^2-x^{16}\left(y^3-x^2\right)
    \end{array}\right.
    \end{equation*}
    
The changes of variables $(w=u,s=u^3-v)$ and $(w=u,v=-s+w^3)$ lead to the equation
 \begin{equation*}
    \begin{array}{rl}
        \hat{f}_3
            &=\left(s^5 -\left(-s+w^3\right)^6\right)^2 +s\left(-s+w^3\right)^{11}.
    \end{array}
\end{equation*}   
 We have computed with \textsf{Singular} \cite{singular} the Tjurina algebras of each of the preceding curves. We see that the Tjurina algebras corresponding to $\hat{f}_1$, $\hat{f}_2$, $\hat{f}_3$ are not isomorphic, therefore the corresponding defining curves are nor analytically equivalent, though they are topologically equivalent since they have the same value semigroup.
\end{ex}

The whole previous discussion and the given examples motivate and encourage us to ask the following.

\begin{question}
Assume that, for some \(k\geq 1\), there are \(\delta\)-sequences \(\Delta_1,\dots,\Delta_k\) such that they have a common associated plane branch semigroup \(\Gamma.\) Is there any relation between the moduli spaces \(M_{\Delta_1},\dots,M_{\Delta_k}\) that can be described in terms of their representation in \(M_\Gamma\)?
\end{question}

\medskip


\bibliographystyle{plain}
\bibliography{bibliodeltadef}
\end{document}